\documentclass{amsart}

\usepackage{amsmath,amsthm,amsfonts,amssymb,latexsym,appendix}
\usepackage{cancel}
\usepackage{tikz}

\usepackage{hyperref}

\newcommand{\<}{\langle}
\renewcommand{\>}{\rangle}

\newtheorem{thm}{Theorem}[section]
\newtheorem*{thm*}{Theorem}
\newtheorem{lemma}[thm]{Lemma}
\newtheorem{conj}[thm]{Conjecture}
\newtheorem{cor}[thm]{Corollary}
\newtheorem*{lemma*}{Lemma}
\newtheorem*{claim*}{Claim}
\newtheorem*{claim1}{Claim 1}
\newtheorem*{claim2}{Claim 2}
\newtheorem*{claim3}{Claim 3}
\newtheorem{prop}[thm]{Proposition}

\newtheorem{question}[thm]{Question}

\theoremstyle{definition}

\newtheorem{defn}[thm]{Definition}

\DeclareMathOperator{\Prob}{Prob}
\DeclareMathOperator{\dom}{dom}
\DeclareMathOperator{\ran}{ran}
\DeclareMathOperator{\diam}{diam}
\DeclareMathOperator{\cl}{cl}

\newcommand{\from}{\colon}

\newcommand{\bfSigma}{\boldsymbol{\Sigma}}

\newcommand{\cantor}{2^{\omega}}

\newcommand{\E}{\mathrel{E}}

\newcommand{\R}{\mathbb{R}}

\newcommand{\F}{\mathbb{F}}

\newcommand{\DC}{\mathsf{DC}}
\newcommand{\AD}{\mathsf{AD}}

\newcommand{\concat}{{}^\smallfrown}

\renewcommand{\subset}{\subseteq}
\renewcommand{\supset}{\supseteq}

\newcommand{\union}{\cup}
\newcommand{\bigunion}{\bigcup}
\newcommand{\inters}{\cap}

\newcommand{\biginters}{\bigcap}

\newcommand{\define}[1]{\emph{#1}} 
\newcommand{\actson}{\curvearrowright}

\newcommand{\cone}{\bigtriangledown}

\begin{document}

\title{On a question of Slaman and Steel}

\thanks{This research was partially carried out while the second author was
visiting Wellington on sabbatical. The second author would like to thank
the first author and Victoria University of Wellington for their hospitality.
The authors would also like to thank Kirill Gura and Antonio Montalb\'an
for helpful comments on preprints of the article. The second author is
partially supported by NSF grant DMS-2348788.}

\author{Adam R. Day}
\address{University of Auckland}
\email{adam.day@auckland.ac.nz}

\author{Andrew S. Marks}
\address{University of California, Berkeley}
\email{marks@math.ucla.edu}

\date{\today}

\begin{abstract}
We consider an old question of Slaman and Steel: whether Turing equivalence
is an increasing union of Borel equivalence relations none of which contain
a uniformly computable infinite sequence. We show this question is deeply
connected to problems surrounding Martin's conjecture and in the theory of
countable Borel equivalence relations. In particular, if Slaman and Steel's
question has a positive answer, it implies there is a universal countable
Borel equivalence relation which is not uniformly universal, that there is a
$(\equiv_T,\equiv_m)$-invariant function which is 
not uniformly invariant on any
pointed perfect set, and which is 
everywhere $\leq_m$-incomparable with both the Turing jump $j$ and its
complement.
\end{abstract}

\maketitle

\section{Introduction}

This paper is a contribution to the study of problems surrounding Martin's conjecture on Turing invariant functions and 
countable Borel equivalence relations. Our central
focus is an old open question of Slaman and Steel which they posed
\cite{SS} in
reaction to their proof in the same paper that 
Turing equivalence is not hyperfinite.
The question they asked is 
 whether Turing
equivalence can be expressed as a union of Borel equivalence
relations $E_n$ where $E_n \subset E_{n+1}$ for all $n$ and so that no $E_n$-class $[x]_{E_n}$ contains an infinite
sequence of reals uniformly computable from $x$. While this seems to be a very
specific question about computability, we show
(Theorem~\ref{thm:hrf_equivalence}) that it is
equivalent to a much more general question of whether every countable Borel
equivalence relation is what we call hyper-Borel-finite
(see Definition~\ref{defn:hef}). 

This question of Slaman and Steel has been completely unstudied since the
1988 paper where it was posed, and it remains open. However, we show that it is deeply connected to problems in both
Borel equivalence relations, and problems surrounding Martin's conjecture. 
In particular,
we show (Corollary~\ref{m_cor}.(1) and (3)) that
if Slaman and Steel's question has a positive answer,
then there is a Borel function $f$  on $\cantor$ that is Turing to many-one invariant  
that is everywhere $\leq_m$-incomparable with both the Turing jump $j$ and $\overline{j}$, and so $f$ is not uniformly invariant on any pointed
perfect set. So under this assumption, the analogues of Steel's conjecture for invariant functions from Turing equivalence to many-one equivalence are false. Such analogues are naturally suggested by 
Kihara and Montalb\'an's work \cite{KM}.
We also show
(Corollary~\ref{m_cor}.(2)) that
if Slaman and Steel's question has a positive answer, then 
many-one
equivalence on $2^\omega$ is a universal countable Borel equivalence
relation. Since many-one equivalence on $2^\omega$ is not 
uniformly universal \cite[Theorem 1.5.(5)]{M}, this implies that if Question~\ref{hyp-rec-fin} has a
positive answer, the conjecture of the second author that every
universal countable Borel equivalence relation is uniformly universal
(\cite[Conjecture 1.1]{M}) is false.

Our main construction is given in Theorem~\ref{m_construction}. This is  
the first result
constructing a nonuniform function between degree structures in
computability theory from any sort of hypothesis. 

Suppose we want to construct a counterexample to part I of Martin's
conjecture. That is, we want to build a Turing invariant function $f \from \cantor \to
\cantor$ such that the Turing degree of $f$ is not constant on a cone, and
$f(x) \ngeq_T x$ on a cone. An obvious strategy is to build $f$ in countably many
stages. At stage $n$, we determine some partial information about $f(x)$
in order to diagonalize against $f(x)$ computing $x$ via the $n$th Turing
reduction. At stage $n$ we also specify how to
``code'' $f(y)$ into $f(x)$ for some of the $y$ such that $y \equiv_T x$ (to ensure
that at the end of the construction, $f$ is Turing invariant). Now consider the
relation $E_n$ where $x \mathrel{E_n} y$ if both $f(x)$ has been coded into
$f(y)$ and $f(y)$ has been coded into $f(x)$ by the $n$th stage of the
construction. Clearly $E_n$ is an equivalence relation, $E_n \subset E_{n+1}$
for all $n$, and Turing equivalence is the union of these equivalence relations:
${\equiv_T} = \bigunion_n E_n$. 

A problem in attempts to construct counterexamples to Martin's
conjecture is that we know essentially nothing about the ways in which
Turing equivalence can be written as an increasing union, apart from Slaman
and Steel's original theorem that Turing equivalence is not hyperfinite. In
particular, it is open whether every way of writing Turing equivalence as an
increasing union ${\equiv_T} = \bigunion_n E_n$ must be trivial in the
sense that there is some $n$ and some pointed perfect set $P$ where $E_n$
is already equal to Turing equivalence, i.e. 
$E_n
\restriction P = (\equiv_T \restriction P)$ (see
Conjecture~\ref{conj:non_approx}). If Conjecture~\ref{conj:non_approx} is
true, attempts to build counterexamples to Martin's conjecture in the way
indicated above seem hopeless.

In the authors' opinion, understanding how Turing equivalence may be
expressed as an increasing union, and Slaman and Steel's
Question~\ref{hyp-rec-fin} seem to be vital steps towards understanding
Martin's conjecture. If Question~\ref{hyp-rec-fin} has a
positive answer, one can hope to improve on the construction in
Theorem~\ref{m_construction} to give a counterexample to Martin's conjecture. 
If Question~\ref{hyp-rec-fin} has a negative answer, perhaps
Conjecture~\ref{conj:non_approx} is true, and there is no nontrivial way of
approximating Turing equivalence from below in countably many stages.

\subsection{Preliminaries}

Our conventions and notation are largely standard. For background on
Martin's conjecture, see \cite{MSS}. For a recent survey of the field of
countable Borel equivalence relations, see \cite{K19}. 

We use lowercase $x,y,z$ to denote elements of $2^\omega$, and $f,g$ for
functions on $2^\omega$.
If $x \in \cantor$, we use $\overline{x}$ to denote the real obtained by
flipping all the bits of $x$ (or the complement of $x$, viewing $x$ as a
subset of $\omega$). If $f \from \cantor \to \cantor$, we similarly use
$\overline{f}$ to denote the function where $\overline{f}(x) = \overline{f(x)}$
for all $x$. If $A \subset \omega$ and $x \in \cantor$, we let $x
\restriction A$ denote the restriction of the function $x$ to $A$.
Equivalently, viewing elements of $2^\omega$ as subsets of $\omega$, $x
\restriction A$ is $x \inters A$. Provided $y \in 2^\omega$ is not the
constant sequence of all 1s, if $A$ is computable, then $x
\restriction A \leq_m y$ if and only if there is a computable function $\rho \from A
\to \omega$ so that for all $n \in A$, $x(n) = y(\rho(n))$. This is because
given such a $\rho \from A \to \omega$, we can fix $n_0$ so $y(n_0) = 0$,
and define $\rho' \from \omega \to \omega$ by $\rho'(n) = \rho(n)$ if $n
\in A$ and $\rho'(n) = n_0$ otherwise. Then $\rho'$ gives a many-one
reduction of $x \inters A$ to $y$. 

Fix a computable bijection $\< \cdot, \cdot \> \from \omega^2 \to
\omega$. We will assume that for all $i, j$ we have $\<i,j\> \geq i$ and
$\<i,j\> \geq j$. If $A \subset \omega$, the \define{$i$th column of $A$}
is $A^{[i]} = \{\<i,j\> \in A \colon j \in \omega\}$. 

\section{Versions of Martin's conjecture for invariant functions from Turing
to many-one degrees}
\label{sec:Tm}

In \cite{KM}, Kihara and Montalb\'an study uniformly degree invariant functions from Turing
degrees to many-one degrees. One of our main results is that if
Slaman and Steel's question has a positive answer, then there is a Borel
degree invariant function $f \colon \cantor \to \cantor$ sending Turing degrees to many-one degrees which is not uniformly
Turing invariant on any pointed perfect set, and is also everywhere $\leq_m$-incomparable with any given countable collection of functions from $\cantor$ to $\cantor$ whose values are all incomputable. 
In this section, we briefly
discuss some open problems around such functions which are suggested by
Kihara and Montalb\'an's work. 

Recall that a function $f \from \cantor \to
\cantor$ is \define{\mbox{$(\equiv_T,\equiv_m)$}-invariant} if $x \equiv_T y$ implies $f(x)
\equiv_m f(y)$. (In the terminology of Borel equivalence relations, we
would say
$f$ is a \define{homomorphism} from $\equiv_T$ to $\equiv_m$.) 
A function $f \from \cantor \to
\cantor$ is \define{uniformly $(\equiv_T,\equiv_m)$-invariant} if there is
a function $u \from \omega^2 \to \omega^2$ so that if $x \equiv_T y$ via
the programs $(i,j)$, then $f(x)
\equiv_m f(y)$ via the programs $u(i,j)$.
If $c \in
\cantor$, then $x
\leq^c_m y$ if there is a function $\rho \from \omega \to \omega$
computable from $c$ so that $x(n) = y(\rho(n))$ for all $n$.
If $f, g \from \cantor
\to \cantor$, then we write $f \leq^\cone_m g$ if there is a Turing cone of
$x$ with base $c$ so that $f(x) \leq^c_m g(x)$. 

Kihara and Montalb\'an show that uniformly $(\equiv_T,\equiv_m)$-invariant functions
are
well-quasi-ordered by $\leq_m^\cone$ and are in
bijective correspondence with Wadge degrees via a simple map they define \cite{KM}. It follows from this bijection with Wadge degrees that the
smallest uniformly $(\equiv_T,\equiv_m)$-invariant functions which are not
constant on a cone are the Turing
jump: $x \mapsto x'$ and its complement $x \mapsto \overline{x'}$, which
are easily seen to correspond to the maps
associated to universal open and closed sets; the lowest nontrivial classes
in the Wadge hierarchy. 

Implicit in Kihara and Montalb\'an's work are obvious analogues of Martin's
conjecture \cite[Conjecture I, II]{SS} and Steel's conjecture
\cite[Conjecture III]{SS} for $(\equiv_T,\equiv_m)$-invariant functions. We
state these conjectures:

\begin{conj}[Martin's conjecture for
$(\equiv_T,\equiv_m)$-invariant functions]\label{conj:M_for_m} Assume $\AD + \DC$. Then 
\begin{enumerate}
\item [I.]
  If $f \from \cantor \to \cantor$ is $(\equiv_T,\equiv_m)$-invariant and the many-one degree
  $[f(x)]_m$ of $f$ is not constant on a Turing cone of $x$, then 
  $f \geq^\cone_m j$, or $f \geq^\cone_m \overline{j}$, where $j(x) = x'$
  is the Turing jump. 
\item[II.] If $f, g \from \cantor \to \cantor$ are $(\equiv_T,\equiv_m)$-invariant, then $f \geq^\cone_m
g$ or $\overline{g} \geq^\cone_m f$.
Furthermore, the order $\leq^\cone_m$ well-quasi-orders the functions on
$\cantor$ that
are $(\equiv_T,\equiv_m)$-invariant.
\end{enumerate}
\end{conj}

\begin{conj}[Steel's conjecture for $(\equiv_T,\equiv_m)$-invariant
functions]\label{conj:steel_for_m}
  Suppose $\AD + \DC$, and suppose $f \from \cantor \to \cantor$ is
  $(\equiv_T,\equiv_m)$-invariant. Then there is a
  uniformly $(\equiv_T,\equiv_m)$-invariant function $g$ so that $f
  \equiv^\cone_m g$. 
\end{conj}

Conjecture~\ref{conj:steel_for_m} implies Conjecture~\ref{conj:M_for_m} by
Kihara and Montalb\'an's work in \cite{KM}.

There is an important
relationship between Turing invariant functions and
\mbox{$(\equiv_T,\equiv_m)$}-invariant functions. Since
$x \leq_T y$ if and only if $x' \leq_m y'$, 
any Turing invariant
function can be turned into a $(\equiv_T,\equiv_m)$-invariant function by
applying the Turing jump. However, because of the parameter $c$ in the
definition of $\leq^\cone_m$, it is not true that if $f'
\geq_m^\cone g'$, then $f(x) \geq_T g(x)$ on a Turing cone of $x$. In
particular, we do not
know whether Conjecture~\ref{conj:M_for_m} and
Conjecture~\ref{conj:steel_for_m} imply Martin's conjecture and Steel's
conjecture. 
However, if we strengthen Conjecture~\ref{conj:steel_for_m} to
use the relation ``$\leq_m$ on a cone'' rather than $\leq^\cone_m$, then we
do obtain a strengthening of Steel's conjecture \cite[Conjecture III]{SS}.

\begin{conj} \label{conj:strong_steel_for_m}
  Suppose $\AD$, and suppose $f$ is $(\equiv_T,\equiv_m)$-invariant. 
  Then there is a
  uniformly $(\equiv_T,\equiv_m)$-invariant function $g$ so that $f(x)
  \equiv_m g(x)$ on a Turing cone of $x$. 
\end{conj}

A standard argument (see the first footnote in \cite{MSS}) shows that if
$f$ is $(\equiv_T,\equiv_m)$-invariant, then $f(x)
  \equiv_m g(x)$ on a cone for some uniformly
  $(\equiv_T,\equiv_m)$-invariant function $g$ if and only if $f$ is itself uniformly
  $(\equiv_T,\equiv_m)$-invariant on a
  pointed perfect set. 

\begin{prop}
  Conjecture~\ref{conj:strong_steel_for_m} implies Steel's conjecture,
  \cite[Conjecture III]{SS}.
\end{prop}
\begin{proof}
  Suppose $f \from \cantor \to \cantor$ is Turing invariant. Then by 
  Conjecture~\ref{conj:strong_steel_for_m}, the map $x \mapsto f(x)'$,
  is uniformly $(\equiv_T,\equiv_m)$-invariant on a pointed perfect set.
  Hence $f$ is uniformly Turing invariant on the same pointed
  perfect set.
\end{proof}

Kihara and Montalb\'an's work is more generally
stated for functions to the space $\mathcal{Q}^\omega$, where $\mathcal{Q}$
is a better-quasi-order. One can more generally ask about the
analogues of the above conjectures for functions to $\mathcal{Q}^\omega$. We
have the following observation due to Kihara and Montalb\'an that the relation $\leq^\cone_m$ cannot be
replaced with ``$\leq_m$ on a cone'' in their work when $\mathcal{Q} \neq
2$:

\begin{prop}[Kihara and Montalb\'an, private communication]
  Suppose $\AD$. Then the $(\equiv_T,\equiv_m)$-invariant functions from
  $2^\omega$ to $3^\omega$ which are not constant on a cone are not
  well-quasi-ordered by the relation ``$\leq_m$ on a cone''.
\end{prop}
\begin{proof}
By
\cite[Theorem 3.6]{M}, many-one reducibility on $3^\omega$ is a uniformly
universal countable Borel equivalence relation. Letting $=_\R$ denote
equality on the real numbers, there is hence a uniform
Borel reduction $f \from \cantor \times \R \to 3^\omega$ from 
 $\equiv_T \times =_\R$ to many-one reducibility on $3^\omega$. For each $y
 \in \R$, the function $f_y(x) = f(x,y)$ is thus a uniformly
 $(\equiv_T,\equiv_m)$-invariant function. Note that if $y \neq y'$, then
 $f_y(x)$ and $f_{y'}(x)$ are not
 $\equiv_m$-equivalent on a cone of $x$, nor are they constant on a cone
 (since $f$ is a Borel reduction). 

 Thus, the relation on Borel functions
 ``$\leq_m$ on a cone'' cannot be a well-quasi-order on the Borel
 uniformly
 $(\equiv_T,\equiv_m)$-invariant functions
 from $2^\omega$ to $3^\omega$, since then it would therefore give a
 well-quasi-order of $\R$. 
\end{proof}

In fact, it is easy to see from the proof of \cite[Theorem 3.6]{M} that for
all $y, y'$ and all $z$, $f_y(z) \ngeq_m f_{y'}(z)$. So all the functions
$f_y$ constructed above are incomparable under $\leq_m$.

It is an open question whether the relation $\leq^\cone_m$ can be
replaced with ``$\leq_m$ on a cone'' in Kihara and Montalb\'an's theorem on
the space $2^\omega$. 

\begin{question}
  Assume $\AD + \DC$. 
  Is there an isomorphism between the Wadge
  degrees and the degrees of the uniformly $(\equiv_T,\equiv_m)$-invariant
  functions under the relation ``$\leq_m$ on a cone''? If 
  $f$ is uniformly $(\equiv_T,\equiv_m)$-invariant and the many-one degree
  $[f(x)]_m$ of $f$ is not constant on a Turing cone of $x$, then is
  $f(x) \geq_m j(x)$ on a cone, or $f(x) \geq_m \overline{j(x)}$ on
  a cone, where $j(x) = x'$
  is the Turing jump? 
\end{question}

\section{Slaman and Steel's Question}
\label{sec:hyper_recursive_finite}

The following notion is essentially due to Slaman and Steel:
\begin{defn}[\cite{SS}]\label{defn:hef}
  Suppose
  $(f_i)_{i \in \omega}$ is a countable sequence of Borel functions 
  $f_i \from X \to X^\omega$. Say that a countable Borel equivalence
  relation $F$ on $X$ is \define{$(f_i)_{i \in \omega}$-finite} if there is
  no $i \in \omega$ and $x \in X$ such that the set 
  $\{f_i(x)(j) \colon j \in \omega\}$ is infinite and $\{f_i(x)(j) \colon j \in
  \omega\} \subset [x]_F$. 
  That is, no 
  $f_i(x)$ is a sequence of infinitely many different elements in the
  $F$-class of $x$. 
  Say that $E$ is
  \define{hyper-$(f_i)_{i \in \omega}$-finite} if there is an increasing
  sequence $F_0 \subset F_1 \subset \ldots$ of Borel subequivalence
  relations of $E$ such that $F_n$ is $(f_i)_{i \in \omega}$-finite for
  every $n$, and $\bigunion_n F_n = E$. Finally, say that $E$
  is \define{hyper-Borel-finite} if for every countable collection of
  Borel functions
  $(f_i)_{i \in \omega}$ where $f_i \from X \to X^\omega$, $E$ is \define{hyper-$(f_i)_{i \in
  \omega}$-finite}.
\end{defn}

Here we can think of each set $\{f_i(x)(j)\}_{j \in \omega}$ as being a potential
witness that some $F$-class is infinite, which we would like to avoid.

Clearly every hyperfinite Borel equivalence relation is
hyper-Borel-finite. It is an open problem to characterize the
hyper-Borel-finite equivalence relations.

\begin{question}
  Is there a non-hyperfinite countable Borel equivalence relation that is
  hyper-Borel-finite?
\end{question}

\begin{question}
  Is every countable Borel equivalence relation hyper-Borel-finite?
\end{question}

Slaman and Steel consider the special case of Definition~\ref{defn:hef} where the function $f_i \from \cantor \to
  (\cantor)^\omega$ gives the columns from the real given by the $i$th Turing reduction $\Phi_i(x)$:
  \[f_i(x)(j) = \begin{cases} \{n \colon \<j,n\> \in \Phi_i(x)\} & \text{ if $\Phi_i(x)$ is total} \\
  x & \text{ otherwise}.
  \end{cases}  
  \] 
  We say that Turing equivalence is \define{hyper-recursively-finite} if $\equiv_T$ is
  hyper-$(f_i)_{i \in \omega}$-finite for the above functions $(f_i)_{i \in
  \omega}$.
  Slaman and Steel posed the question of whether Turing equivalence is
  hyper-recursively-finite in \cite[Question 6]{SS}, though in the setting of
  $\AD$ rather than just for Borel functions. We work in the Borel setting
  because it makes the statements of some of our theorems more
  straightforward. However,
  all the arguments of the paper can be adapted to the setting of $\AD$ as
  usual.

\begin{question}[\cite{SS}]\label{hyp-rec-fin}
  Is Turing equivalence hyper-recursively-finite?
\end{question}

  This problem about Turing equivalence is equivalent to the more
  general problem of whether every countable Borel equivalence relation is
  hyper-Borel-finite. This self-strengthening property of
  hyper-recursive-finiteness of $\equiv_T$ will be an essential ingredient in
  our proof of Theorem~\ref{m_construction}.

\begin{thm}\label{thm:hrf_equivalence}
  The following are equivalent:
  \begin{enumerate}
    \item $\equiv_T$ is hyper-recursively-finite.
    \item Every countable Borel equivalence relation $E$ is
    hyper-Borel-finite.
  \end{enumerate}
\end{thm}
\begin{proof} 
(1) is a special case of (2), and is hence implied by it. We prove that (1)
implies (2). Fix a witness
  $F_0 \subset F_1 \subset \ldots$  that $\equiv_T$ is
  hyper-recursively-finite. We wish to show that every countable Borel
  equivalence relation $E$ is hyper-Borel-finite.
  We may assume that $E$ is a countable Borel equivalence relation on
  $\cantor$. We may further suppose that $E$ is $\Delta^1_1$ and $(f_i)_{i
  \in \omega}$ is uniformly $\Delta^1_1$; our proof relativizes. 

  Since $E$ is a $\Delta^1_1$ relation with countable vertical sections,
  and $(f_i)_{i \in \omega}$ is uniformly $\Delta^1_1$, there is some
  computable ordinal notation $\alpha$ such that for all $x \in \cantor$
  and for all $y \E x$, $x^{(\alpha)} \geq_T y$ , and $x^{(\alpha)} \geq_T
  \bigoplus_{i \in \omega} f_i(x)$. Now if we let $\beta = \omega \cdot
  \alpha$, then $x^{(\alpha)} \geq_T y$ implies $x^{(\beta)} \geq_T
  y^{(\beta)}$. Hence, if $x \E y$, then $x^{(\beta)} \equiv_T
  y^{(\beta)}$, and $x^{(\beta)} \geq_T \left(\bigoplus_{i \in \omega}
  f_i(x) \right)^{(\beta)}$. Note that the function $x \mapsto x^{(\beta)}$
  is injective.
  
  Define $E_k$ by \[x \mathrel{E_k} y \iff x E y \land x^{(\beta)} F_k
  y^{(\beta)}.\] We claim that $(E_k)_{k \in \omega}$ witness that $E$
  is hyper-$(f_i)$-finite. Suppose not. Then there exists $E_k$, $x$ and
  $i$ such that $\{f_i(x)(j) \colon j \in \omega\}$ is infinite and $x
  \mathrel{E_k} f_i(x)(j)$ for all $j \in \omega$. This implies $x^{(\beta)} \mathrel{F_k}
  (f_i(x)(j))^{(\beta)}$ for all $j$ by definition of $E_k$. Now the sequence
  $\left((f_i(x)(j))^{(\beta)} \right)_{j \in \omega}$ is uniformly
  recursive in $x^{(\beta)}$ since 
  $x^{(\beta)} \geq_T (f_i(x))^{(\beta)}$. The set
  $\left\{(f_i(x)(j))^{(\beta)} \colon j \in \omega \right\}$ is still
  infinite since the jump operator $x \mapsto x^{(\beta)}$ is injective. This contradicts
  that $(F_k)_{k \in \omega}$ is a witness that $\equiv_T$ is
  hyper-recursively-finite.
  \end{proof}

The key in the above proof is that given any countable Borel equivalence
$E$ on $X$ and Borel functions $(f_i)$ from $X \to X^\omega$, we can find
an injective Borel
homomorphism $h$ from $E$ to $\equiv_T$ so that the image of each $f_i$
under $h$ is a computable function. Similar theorems to
Theorem~\ref{thm:hrf_equivalence} are true for other weakly universal
countable Borel equivalence relations, and collections of ``universal''
functions with respect to them. For example, let $E_\infty$ be the
orbit equivalence relation of the
shift action of the free group $\F_{\omega} = \<\gamma_{i,j}\>_{i,j \in \omega}$ on
$\omega^{\F_\omega}$ (so we are
indexing the generators of $\F_{\omega}$ by elements of $\omega^2$). 
Let 
$f_i(x)(j) = (\gamma_{i,j} \cdot x)$. Then $E_\infty$ is
hyper-$(f_i)$-finite if and only if every countable Borel equivalence
relation is hyper-Borel-finite. 

Boykin and Jackson have introduced the class of Borel bounded
equivalence relations \cite{BJ}. For these equivalence relations it is an open
problem
whether there is some non-hyperfinite Borel bounded equivalence
relation, and also whether all Borel equivalence relations are Borel
bounded. Similarly both these problems are open for the hyper-Borel-finite Borel equivalence relations. We pose the question of whether there
is a relationship between $E$ being
hyper-Borel-finite and being Borel bounded.

\begin{question}
  Is every Borel bounded countable Borel equivalence relation hyper-Borel-finite?
\end{question}
 
Straightforward measure theoretic and Baire category arguments
cannot prove that any countable Borel equivalence relation is not
hyper-Borel-finite. This follows for Baire category from generic
hyperfiniteness. 
To analyze hyper-Borel-finiteness in the measure theoretic setting, we
first need an easy lemma about functions selecting subsets of a finite set. 
Below, $\Prob(X)$ indicates the probability of an event $X$. 

\begin{lemma}\label{prob_pigeon}
  Suppose $(X,\mu)$ is a standard probability space, $k \leq n$, $Y$ is a
  finite set where $|Y| = n$,
  and $g
  \from X \to [Y]^k$ is any measurable function associating
  to each $x \in X$ a subset of $Y$ of size $k$.
  Then for any $m \geq 1$, there is a set $S \subset Y$ with $|S| \leq m$ such
  that $\Prob(g(x) \inters S \neq \emptyset) \geq 1 - \left(1 - k/n\right)^m$.
\end{lemma}

The point of the lemma for us is the case where $0 \ll k \ll n$, and $m = \lceil
\frac{n}{\sqrt{k}}
\rceil$. Think of $g$ as being a probabilistic process for choosing $k$
elements out of our set $Y$ of size $n$. Then the lemma says 
we can choose $S \subset Y$ of size $|S| \leq m$ such that $\Prob(S
\inters g(x) \neq \emptyset)$ is close to $1$. That is, we can find a
``small'' $S$ (of size much less than $|Y| = n$)
so that with very high probability, one of the $k$ elements we
choose using the process $g$ comes from $S$.
This is because $(1 -
k/n)^{n/k} \approx 1/e$, so $(1 - k/n)^{m} \approx
(1/e)^{\sqrt{k}} \approx 0$. 

\begin{proof}
  If we select $i$ from $Y$ uniformly at random, and $x$ from $X$ at
  random (wrt $\mu$),
  then $\Prob(i \notin g(x)) = 1-k/n$, since $g(x)$ has $k$
  elements. So if we pick $m$ elements $i_1, \ldots, i_{m}$ from $Y$ 
  uniformly at random (allowing repetitions in the list), and let $S = \{i_1, \ldots, i_{m}\}$, then $\Prob(S
  \inters g(x) = \emptyset) = \left(1 - k/n \right)^{m}$. Hence,
  there must be some fixed set $S = \{i_1, \ldots, i_m\}$ such that
  $\Prob(g(x) \inters S = \emptyset) \leq \left(1 - k/n \right)^{m}$,
  and so \mbox{$\Prob(g(x) \inters S \neq \emptyset) \geq 1 - \left(1 - k/n
  \right)^{m}$}. (It is possible that $|S| < m$ if we have repetitions). 
  %If this $S^*$ has fewer than $m$ elements (since $i_{j} =
  %i_l$ for some $j,l$), then we can arbitrarily add elements to $S^*$ to
  %make it have size $m$, and this probability only increases.
\end{proof}

We now have the following theorem analyzing hyper-Borel-finiteness in the
measure theoretic setting:
\begin{thm}
  Suppose $E$ is a countable Borel equivalence relation on a standard Borel
  space $X$, $(f_i)_{i \in \omega}$ are Borel functions from $X$ to
  $X^\omega$, and $\mu$ is a Borel probability measure on $X$. Then there
  is a $\mu$-conull Borel set $B$ so that $E \restriction B$ is
  hyper-$(f_i)$-finite. 
\end{thm}
\begin{proof}
  We claim that for any $\epsilon > 0$, and any single Borel
  function $f \from X \to X^\omega$, there is a Borel set $A \subset X$
  with $\mu(A) > 1 - \epsilon$ such that $E \restriction A$ is
  $f$-finite. (By $f$-finite for a single $f$, we mean that no $E
  \restriction A$-class contains an infinite set of the form $\{f(x)(j)
  \colon j \in \omega\}$).

  The theorem follows easily from this claim. Choose a sequence of positive
  real numbers $(a_{i,n})_{i,n \in \omega}$ so that $\sum_{i,n} a_{i,n} < \infty$.
  Then for each $i$ and $n$, let $A_{i,n} \subset X$ be a Borel set so that
  $E \restriction A_{i,n}$ is $f_i$-finite (just for the single
  function $f_i$), and $\mu(A_{i,n}) > 1 - a_{i,n}$. Then let $B_m =
  \biginters_{n \geq m \land i \in \omega} A_{i,n}$. Since $B_m \subset
  A_{i,m}$ for every $i$, $E \restriction B_m$ is $(f_i)_{i \in
  \omega}$-finite (for the entire sequence of $(f_i)_{i \in \omega}$). The
  $B_m$ are increasing sets. We have $\mu(B_m) > 1 - \sum_{n \geq m \land i \in
  \omega} a_{i,n}$, so $\mu(B_m) \to 1$. Let $A = \bigunion_m B_m$. Then $E
  \restriction A$ is hyper-$(f_i)$-finite as witnessed by $E \restriction
  B_m$.

  We prove the claim.  
  Fix a Borel function $f \from X \to X^\omega$.
  Without loss of generality we may assume that $\{f(x)(j) \colon j \in
  \omega\}$ is infinite for every $x$, and that $\mu$ has no atoms. The idea here is to use Lemma~\ref{prob_pigeon} to
  find a set $A$ of measure $\mu(A) > 1 - \epsilon$ such that for every $x \in
  A$, there is some $j$ such that $f(x)(j) \notin A$.

  We may assume by the isomorphism theorem for standard probability spaces
  that $X = 2^\omega$ and $\mu$ is Lebesgue measure. Consider the function $U_l(x) = \{N_s \colon s \in
  2^l \land (\exists j) f(x)(j) \in N_s \}$. That is, $U_l(x)$ is the
  collection of basic open neighborhoods $N_s$, where $s$ has length $l$,
  such that $N_s$ contains some element of the sequence $f(x)$. Since the neighborhoods $N_s$
  separate points, for every $x$ we have $|U_l(x)| \to \infty$ as $l \to
  \infty$. Letting $X_{l,k} = \{x \in X \colon 
  |U_l(x)| \geq k\}$,
  we may choose a sufficiently large $l$ so that $\mu(X_{l,k}) > 1 - \epsilon$.

  Now by picking $l \gg k \gg 0$ sufficiently large
  and applying Lemma~\ref{prob_pigeon} to
  the function selecting the least $k$ elements of $U_{l}(x)$, 
  we can choose a set $S \subset \{N_s \colon s \in
  2^l\}$ of size $|S| <  2^l/\sqrt{k}$  so that
  $\mu( \{x \in X_{l,k} \colon U_l(x)
  \inters S \neq \emptyset\})$ is arbitrarily close to $\mu(X_{l,k})$.
  Note that $\mu(\bigunion S) < \frac{1}{\sqrt{k}}$.
  Let 
  \[A = \{x \in X_{l,k} \setminus \bigunion S \colon \exists i f(x)(i) \in \bigunion
  S\}\]
  The claim follows.
\end{proof}

The above proof is trivial in the sense that the subequivalence relations
witnessing hyper-$(f_i)$-finiteness are simply the original equivalence
relation restricted to some Borel subset of $X$. This style of witness that
an equivalence relation is hyper-Borel-finite cannot work in general to
show that an equivalence relation is hyper-Borel-finite. For example, there
is no increasing sequence of Borel sets $(A_k)_{k \in
\omega}$ such that $2^\omega = \bigunion_k A_k$, and the equivalence
relations $\equiv_T \restriction
A_k$ witness that $\equiv_T$ is hyper-recursively-finite. To see this,
note that some $A_n$ must contain a pointed perfect set, and hence
$\equiv_T \restriction A_n$ must contain a uniformly computable infinite sequence.

%We finish this section by discussing hyper-Borel-finiteness and
%compressibility. Let $I_\omega$ be the equivalence relation on $\omega$
%where all elements are equivalent. Suppose $E$ is a countable Borel
%equivalence relation on $X$. Then let $E \times I_\omega$ be the countable
%Borel equivalence relation on $X \times \omega$ where $(x,y) \mathrel{E
%\times I_\omega} (x',y')$ if $x \mathrel{E} x'$ and $y \mathrel{I_\omega}
%y'$. Then $E \times I_\omega$ is a countable Borel equivalence
%relation, and $E \sim_B E \times I_\omega$. One advantage of $E \times
%I_\omega$, however, is that it is generated by the action of a finitely
%generated group. Suppose $a \from \Gamma \actson X$ is an action of a
%countable group $\Gamma$ generating $E$, and let $\gamma_0, \gamma_1,
%\ldots$ be an enunmeration of $\Gamma$. Then $E \times I_\omega$ is
%generated by an action of the free group $\F_2 = \<a,b\>$ on two
%generators. Let $\rho \from \omega \to \omega$ be a bijection with a
%single orbit (i.e. for all $n, m \in \omega$ there exists some $k \in Z$ so
%that $\rho^k(n) = m)$. Then we can define an action of $\F_2$ on $X \times
%\omega$ by $a \cdot (x,n) = (\gamma_n \cdot x, n)$, and $b \cdot (x,n) =
%(x,\rho(n))$. Clearly this action generates $E \times I_\omega$. Now we
%note the following:
%\begin{lemma}
%  $E$ is hyper-Borel-finite if and only if $E \times I_\omega$ is hyper-Borel-finite.
%\end{lemma}
%\begin{proof}
%$\Leftarrow$ is trivial since $E$ is a subequivalence relation of $E \times
%I_\omega$. 
%$\Rightarrow$ follows since 
%
%\end{proof}

\section{Strengthenings of the Kuratowski-Mycielski theorem}

Two often used constructions in computability theory are 
\begin{enumerate}
\item There is a Borel function $f \from \cantor \to
\cantor$ so that if $x_0, \ldots, x_n$ are distinct, then 
$f(x_0), \ldots f(x_n)$ are mutually $1$-generic.
\item There is a Borel function $f \from \cantor \to \cantor$ so that for
all $x$, $f(x)$ is $x$-generic.
\end{enumerate}
(1) is true since there is a perfect tree whose infinite paths are mutual $1$-generics (hence $f$ in (1)
may be continuous). (2) is true since
$x'$ can compute an $x$-generic real uniformly, and so $f$ in this case may
be Baire class $1$ (i.e. $\bfSigma^0_2$-measurable).

It is impossible 
to have a function
$f$ with both properties (1) and (2):

\begin{prop}
  There is no Borel function $f \from \cantor \to \cantor$ so that:
  \begin{enumerate}
    \item If $x_0, \ldots, x_n$ are distinct, then $f(x_0), \ldots f(x_n)$ are
    mutually $1$-generic.
    \item For all $x$, $f(x)$ is $x$-generic.
  \end{enumerate}
\end{prop}
\begin{proof}
  If (2) holds, then $\ran(f)$ is nonmeager. This is true
  because if $\ran(f)$ is meager, the complement of $\ran(f)$ is comeager and hence it
  would contain a dense
  $G_\delta$ set $A$ which is coded by some real $z$. But since $f(z)$
  is $z$-generic, $f(z) \in A$, and so $f(z) \notin \ran(f)$. 

  Now $\ran(f)$ is $\bfSigma^1_1$ and so it has the Baire property. Since $\ran(f)$
  is nonmeager, it is therefore comeager in some basic open set $N_s$. But this
  implies that $\ran(f)$ contains two elements $f(x_0) \neq f(x_1)$ which
  are equal mod finite and hence are not mutually $1$-generic.
\end{proof}

The point of this section is to prove 
Lemma~\ref{genericity_lemma}  where we make (1) above
compatible with a weakening of (2). 
Instead of $f(x)$ being
$x$-generic, we can make $f(x)$ and $x$ a minimal pair in some suitable sense. 
We will need such functions to prove part (3) of Theorem~\ref{m_construction}.
The reader who is not interested in that part of the Theorem may skip this
section, and simply use a continuous map to a set of mutual $1$-generics instead
of the functions constructed in this section. 

\begin{lemma}\label{genericity_lemma}
  Suppose $E$ is a countable Borel equivalence relation on $\cantor$, 
  and $(g_i)_{i \in \omega}$ is a countable collection of Borel functions from
  $\cantor$ to $\cantor$ so $g_i(x)$ is incomputable for all $x$. Then there is a
  Borel function $f \from \cantor \to \cantor$ such that
  \begin{enumerate}
    \item[(1)] If $x_0, \ldots, x_n$ are distinct, then $f(x_0), \ldots f(x_n)$ are
    mutually $1$-generic.
    \item[(2)] For all $x,y \in \cantor$ such that $x \mathrel{E} y$, there
    is no $z$ and $i$ so that $z \leq_m g_i(x)$ and $z \leq_m f(y)$ via a
    many-one reduction with infinite range.
  \end{enumerate}
\end{lemma}

Note that since $f(y)$ is $1$-generic, if $z \leq_m f(y)$ via a
many-one reduction with infinite range, then $z$ is not computable. 

This lemma follows easily from the more general
Lemma~\ref{general_genericity_lemma} below:
\begin{proof}[Proof of Lemma~\ref{genericity_lemma}:]
  Apply Lemma~\ref{general_genericity_lemma} where $Y = Z = \cantor$,
  $C_n \subset (\cantor)^n$ is the set of mutually $1$-generic $n$-tuples, 
  $y \mathrel{S_1} z$ if $y \geq_m z$ via a many-one reduction with infinite
  range, and $x \mathrel{R} z$ if $g_i(x) \geq_m z$ for some $i$. 
  Note that if $\rho \from \omega \to \omega$ is a many-one reduction with
  infinite range, then there is no $z \in \cantor$ such that $y \geq_m z$ via $\rho$
  for nonmeagerly many $y$, since the set of $y$ such that there exists an
  $n$ such that $n \in z \iff \rho(n) \notin y$ is open and dense. 
\end{proof}

We now prove the following strengthening of the 
Kuratowski-Mycielski theorem \cite[Theorem 19.1]{K95}. 
Say that a relation $R \subset X \times Y$ has \define{countable vertical
sections} if for all $x \in X$ there are countably many $y \in Y$ such that
$x \mathrel{R} y$.

\begin{lemma}\label{general_genericity_lemma}
  Suppose $E$ is a countable Borel equivalence relation on a Polish space $X$. Let $Y,Z$ be
  Polish spaces and $R \subset X \times Z$ and $S_n \subset Y^n \times Z$ 
  be Borel relations with countable vertical sections.
  Then for any collection
  $(C_n)_{n \in \omega}$
  of comeager sets $C_n \subset Y^n$, there is a Borel injection $f \from X
  \to Y$ such that 
  \begin{enumerate}
    \item For all $x_1, \ldots, x_n \in
    X$, $(f(x_1), \ldots, f(x_n)) \in C_n$.
    \item For all $x \in X$ and distinct $x_1, \ldots, x_n \in [x]_E$, if $x
    \mathrel{R} z$
    and \mbox{$(f(x_1), \ldots, f(x_n)) \mathrel{S_n} z$}, then there is a nonmeager
    set of $\vec{y} \in Y^n$ such that $\vec{y} \mathrel{S_n} z$. 
  \end{enumerate}
\end{lemma}
Roughly this says that there is a Borel function $f$ so that any finitely
many elements of $\ran(f)$ are ``mutually
generic'' (i.e. in $C_n$), and that if $x_1, \ldots, x_n \in [x]_E$, then $x$ and $(f(x_1),
\ldots, f(x_n))$ form a ``minimal pair'' (with respect to $R$ and $S_n$).
\begin{proof}
  Fix countable bases $\mathcal{B}_X,\mathcal{B}_Y,\mathcal{B}_Z$ of $X, Y$, and $Z$.
  Also fix a complete
  metric $d$ generating the topology of $Y$. Say that an
  \define{approximation} $p$ of $f$ 
  is a function $p \from P \to \mathcal{B}_Y$ where $P$ is a Borel partition of $X$ into finitely many Borel sets.
  Say that
  an approximation $p' \from P' \to \mathcal{B}_Y$ refines
  $p \from P \to \mathcal{B}_Y$ if $P'$ refines
  $P$, and if $A' \in P'$ and $A \in P$ are
  such that $A' \subset A$, then $p'(A') \subset p(A)$. 

  Suppose that $p_0, p_1, \ldots$ is a sequence of approximations where 
  $p_{n+1}$ refines $p_{n}$, 
  \begin{enumerate}
  \item[(a)] $\max \{\diam(U) \colon U
  \in \ran(p_n)\} \to 0$ as $n \to \infty$, and
  \item[(b)] for all $n$, there exists $m > n$, so that $A \in \dom(p_n)$, $A'
  \in \dom(p_m)$ and $A' \subset A$ implies $\cl(p_m(A')) \subset p_n(A)$,
  where $\cl$ denotes closure.
  \end{enumerate}
  Then we can associate to
  this sequence the function $f \from X \to Y$ where $f(x) = y$ if $\{y\} =
  \biginters_n p_n(A_{x,n})$ where
  $A_{x,n}$
  is the unique element of $\dom(p_n)$ such that $x \in A_{x,n}$. Conditions
  (a) and (b) ensure that $\biginters_n p_n(A_{x,n})$ is 
  a singleton for every
  $x$.
  We will
  construct $f$ in this way, where the sequence $(p_i)_{i \in \omega}$ is a sufficiently generic
  sequence of approximations. Clearly 
  (1) in the
  statement of the Lemma
  will be true for a sufficiently generic sequence. We give a density argument to
  justify why (2) will be true.

  Since $R$, $S_n$ have countable vertical sections, by Lusin-Novikov
  uniformization \cite[18.5]{K95}, there are Borel functions
  $(g_i)_{i \in \omega}$ and $(h_{n,i})_{i,n \in \omega}$ where $g_i
  \from X \to Z$ and $h_{n,i} \from Y^n \to Z$ such that $x \mathrel{R} z$
  if and only if $g_i(x) = z$ for some $i$, and $\vec{y} \mathrel{S_n} z$ iff
  $h_{n,i}(\vec{y}) = z$ for some $i$.
  By perhaps refining the sets
  $C_n$, we may assume that the functions $h_{n,i}$ are continuous on
  $C_n$, since any Borel function is continuous on a comeager set
  \cite[Theorem 8.38]{K95}.
  By the Feldman-Moore theorem, we may fix a Borel action of a countable
  group $\Gamma$ generating $E$.
  Let $\mathcal{G}$ be the set of $z \in Z$ such that for some $n$, there
  is a nonmeager set of $\vec{y} \in Y^n$ such that $\vec{y} \mathrel{S_n}
  z$.

  Fix an approximation $p$, finitely many disjoint basic open sets
  $V_{1}, \ldots, V_{n} \subset X$ and group elements $\gamma_1,
  \ldots, \gamma_n \in \Gamma$, and $j, k \in \omega$. It suffices to show
  that we can refine $p$ to an approximation $p^*$ such that for all $x \in
  X$, if $\gamma_i \cdot x \in V_{i}$ for all $i \leq n$, then
  either 
  \[\tag{*} (h_{n,k} \restriction C_n)(p^*([\gamma_1 \cdot x]) \times \ldots \times
  p^*([\gamma_n \cdot x])) \subset \mathcal{G}, \text{ or}\]
  \[\tag{**} 
  g_j(x) \notin (h_{n,k} \restriction C_n)(p^*([\gamma_1 \cdot x]) \times \ldots \times
  p^*([\gamma_n \cdot x]))\]
  where by $[\gamma_i \cdot x]$ we mean the element of $\dom(p^*)$ that
  contains $\gamma_i \cdot x$. That is, the condition above is that if
  $\gamma_i \cdot x \in V_i$ for all $i \leq n$, then the value of $h_{n,k}(f(\gamma_1 \cdot x), \ldots, f(\gamma_n
  \cdot x))$ is
  ``forced" by $p^*$ to be in $\mathcal{G}$, or forced to be different from $g_j(x)$.

  Let $B = \{x \colon (\forall i \leq n) \gamma_i \cdot x \in V_i\}$. These
  are the $x$ for which we must ensure that either (*) or (**) holds. 
  Let $P = \dom(p)$. By refining the domain of $p$, we may assume that
  every element of $P$ is
  either contained in or disjoint from 
  $\gamma_i \cdot B$ for every $i \leq n$. By similarly refining the
  domain, we may furthermore assume that if $A \in P$ is such that $A \subset \gamma_i \cdot
  B$, then $\gamma_{i'} \gamma_{i}^{-1} \cdot A \in P$ for all $i' \leq n$.

  We now define $p^*$. For all $A \in P$ such that $A \nsubseteq \gamma_i
  \cdot B$ for all $i \leq n$, put $A \in \dom(p^*)$, and define $p^*(A) =
  p(A)$. Any remaining $A \in P$ belongs to a tuple 
  $(A_1, \ldots, A_n)$ of elements of $P$ where 
  $A_i \subset \gamma_i \cdot B$ for all $i \leq n$ and $A_{i'} =
  \gamma_{i'} \cdot \gamma_{i}^{-1} \cdot A_i$ for all $i, i' \leq n$ (by
  our assumption on $P$ from the previous paragraph). So for all $x$, if $\gamma_i
  \cdot x \in A_i$ for some $i \leq n$, then $\gamma_i \cdot x \in A_i$ for all
  $i \leq n$. We will define $p^*$ on these $A_i$ to satisfy (*) or (**).
  Letting $U_i = p(A_i)$ for every $i \leq n$, we ask if there are basic
  open sets $U_{i}', U_{i}'' \subset U_{i}$ and disjoint basic open sets
  $W', W'' \subset Z$ so that $(h_{n,k} \restriction C_n)(U_{1}', \ldots,
  U_{n}') \subset W'$ and $(h_{n,k} \restriction C_n)(U_1'', \ldots,
  U_{n}'') \subset W''$.

  Case 1: if such $W'$ and $W''$ do not exist, then put $A_i \in \dom(p^*)$ and define $p^*(A_i) = p(A_i) =
  U_{i}$ for every $i \leq n$. Since $h_{n,k} \restriction C_n$
  is continuous, then $(h_{n,k} \restriction C_n)(U_{1}, \ldots, U_{n})$
  must be a singleton, which must therefore be in $\mathcal{G}$. So in this
  case (*) is satisfied for all $x$ such that $\gamma_i \cdot x \in A_i$
  for $i \leq n$.

  Case 2: if such $W'$ and $W''$ do exist, let
  $A_i' = \{x \colon g_j(\gamma_i^{-1} \cdot x) \in W'\}$, and for every $i
  \leq n$, put both $A_i'$ and $A_i \setminus A_i'$ in $\dom(p^*)$, and
  define $p^*(A_i') = U_{i}''$, and $p^*(A_i \setminus A_i') = U_{i}'$.
  Then by definition, (**) holds for every $x$ such that $\gamma_i \cdot x
  \in A_i$ for $i \leq n$. 
\end{proof}

We remark that there are interesting 
open problems about the extent to which the
Kuratowski-Mycielski theorem can be generalized. For example,
\begin{question}
  Does there exist a Borel function $f \from \cantor \to \cantor$, so that
  for all distinct $x, y$ with $x \leq_T y$, $f(x)$ and $f(y)$ are mutually $x$-generic? 
\end{question}

\section{A nonuniform construction}

In our main construction in the proof of Theorem~\ref{m_construction}, we
will do coding using countably many computable injections
$c_m \from \omega \to \omega$ with disjoint ranges. Precisely, we will ensure that if $x
\mathrel{E} y$, then $f(x) \leq_1 f(y)$ via one of these one-one reductions $c_m$. We will want this coding method to be very ``generic'' so that each $c_m$ is computable, but the sequence $(c_m)_{m \in \omega}$ is not uniformly computable, and for example, if $A \subset \bigunion_m \ran(c_m)$ is c.e., then $A \inters \ran(c_m)$ is infinite for some single $c_m$ (so there is no c.e.\ way of picking finitely many elements of $\ran(c_m)$ for each $m$). 

We begin this section with some definitions and lemmas related to the kind of
coding we will do. The reader may want to read the first few paragraphs of the
proof of Theorem~\ref{m_construction} up to the definition of $f$ and
verification of (1) to motivate these definitions.

\begin{defn}\label{defn:coding_decoding}
Suppose $(c_m)_{m \in \omega}$ is a sequence of injections $c_m \colon \omega \to \omega$ with disjoint ranges such that $c_m(n) > n$ for all $n$. 
We define the \define{decoding function} $d \from \omega \to \omega^{<
\omega}$ associated to $(c_m)_{m \in \omega}$ as
follows: 
\[d(n) = \begin{cases}
\emptyset & \text{ if $n \notin \ran(c_m)$ for any $m$}\\
m \concat d(c_m^{-1}(n)) & \text{ if $n \in \ran(c_m)$}
\end{cases}
\]
Where $\emptyset$ denotes the empty string and $\concat$ denotes 
concatenation of strings. 
Similarly, define $d_s \from \omega \to \omega^{< \omega}$ in the same way but where we only use $c_m$ where
$m \leq s$.
\[d_s(n) = \begin{cases}
\emptyset & \text{ if $n \notin \ran(c_m)$ for any $m \leq s$}\\
m \concat d_s(c_m^{-1}(n)) & \text{ if $n \in \ran(c_m)$ and $m \leq s$}
\end{cases}
\]
Finally, define $b, b_s \from \omega \to \omega$ as follows.
Define $b(n) = (c_{m_0} \circ \ldots \circ c_{m_k})^{-1}(n)$
where $m_0, \ldots, m_k$ are such that $d(n) = (m_0, \ldots, m_k)$.
Similarly, 
$b_s(n) = (c_{m_0} \circ \ldots \circ c_{m_k})^{-1}(n)$
where $m_0, \ldots, m_k$ are such that $d_s(n) = (m_0, \ldots, m_k)$. We use the
convention that the empty composition is the identity. So in particular if
$d(n)=\emptyset$, then $b(n)=n$. 
\end{defn}

We can think of $d$ in the following way. Any $n \in \omega$
can be in the range of at most one $c_m$ since the $(c_m)_{m \in \omega}$
have disjoint ranges. If $n$ is in the range of some $c_m$, the number
$c_m^{-1}(n)$ is strictly less than $n$. 
Iterating this process, there is a unique longest sequence $m_0, \ldots,
m_k$ so that $n \in \ran(c_{m_0}\circ \ldots \circ c_{m_k})$. This longest
such sequence $(m_0, \ldots, m_k)$ is defined to be $d(n)$. The function
$d_s$ is defined the same way but where we restrict to only considering
$c_m$ with $m \leq s$. Finally, $b$ and $b_s$ are the functions which map $n$ to the
number obtained by repeatedly taking the inverse image of $n$ under $c_{m_0},
\ldots, c_{m_k}$ where $(m_0, \ldots, m_k)$ is either $d(n)$ or $d_s(n)$
respectively. 
Note that $d_s(n)$ is an initial segment of $d(n)$, and in fact $d(n) =
d_s(n) \concat d(b_s(n))$ for every $n, s$. 
We also have that $\ran(b_s)$ is the complement of $\ran(c_0) \union \ldots \union \ran(c_s)$, and $\ran(b)$ is the complement of $\bigunion_{m \in \omega} \ran(c_m)$.

We now describe the functions $(c_m)_{m \in \omega}$ we will use in the
proof of Theorem~\ref{m_construction}. Below if $t \in \omega^{< \omega}$
is a sequence, then $\max t$ denotes the largest number in the sequence
$t$. We take the convention that $\max \emptyset = 0$. Recall that $A^{[i]}
= \{\<i,j\> \in A \colon j \in \omega\}$ is the $i$th column of $A$. 

\begin{lemma}
\label{lem:coding_lemma}
  There is a sequence $(c_m)_{m \in \omega}$ of injective 
  computable functions $c_m \from \omega \to \omega$ with disjoint ranges so that $c_m(n) > n$ for all $n,m$, and an 
  infinite computable set $D_0 \subset \omega$ so that $D_0$
  is disjoint from $\bigunion_{m \in \omega}\ran(c_m)$, and 
\begin{enumerate}
  \item For all computable $\rho \from \omega \to \omega$, there exists an $s \in \omega$ so that either $b_s (\rho(\omega))$
  is finite, or there exists a computable infinite set $B$ so that for all
  $n \in B$, $\max
  d(\rho(n)) \leq s$, and $b_s(\rho(B))$ is infinite. 

  \item For all computable $\rho \from \omega \to \omega$, there exists an $s \in \omega$ so that either for infinitely many
  $i$, $b_s(\rho(\omega^{[i]}))$ is finite, or there is a computable set $B$ 
  so that for all $n \in B$, $\max d(\rho(n)) \leq s$, and for
  all but finitely many $i$, $b_s(\rho(B^{[i]}))$ is infinite.
  \end{enumerate}
\end{lemma}
\begin{proof}
Suppose $\rho \from \omega \to \omega$ is computable and $\rho' \from 
\omega \to \omega$ is defined by $\rho'(\<i,j\>) = \rho(j)$, so $\rho'$
copies the values of $\rho$ on every
column of $\omega$. Then if (2) holds for $\rho'$ then (1) holds for
$\rho$. So we only need to verify property (2).

We construct the sequence $(c_m)_{m \in \omega}$ in countably many stages.
We will also
build a sequence $(D_m)_{m \in \omega}$ of subsets of $\omega$ where for
all $m$, $D_{m} \supset D_{m-1}$, $D_m$ is disjoint from $\ran(c_0) \union \ldots \union \ran(c_m)$, and
$\ran(c_0) \union \ldots \union \ran(c_m) \union D_m$ is coinfinite. Though
each $c_m$ and $D_m$ will be computable, neither the sequence $(c_m)_{m \in
\omega}$ nor $(D_m)_{m \in \omega}$ will be uniformly computable.

Note also that since $c_m(n)> n$ for all $n$, if $c_m$ is computable, then  
$\ran(c_m)$ is computable, and $c_m^{-1}$ is computable since $k \in \ran(c_m)$ if there exists some $n < k$ such that $c_m(n) = k$. Hence for each $s$, the functions $d_s$ and $b_s$
will be computable.

Let $c_0$ be any computable function such that $c_0(n) > n$ for all $n$, and $D_0 \subset
\omega$ be any computable infinite set so that $D_0$ and $\ran(c_0)$ are
disjoint and $D_0 \union \ran(c_0)$ is coinfinite. Since we will ensure
that $\ran(c_m)$ is disjoint from $D_m \supset D_0$ for every $m$, the
required property that $D_0$ will be disjoint from $\union_{m} \ran(c_m)$
will be true at the end of the construction. 

At stage $s$, let $\rho \colon \omega \to \omega$ be the $s$th total
computable function, and suppose we have defined $c_s$ and $D_s$. We will
define $c_{s+1}$ and $D_{s+1}$ so that (2) is true. We may
assume that there is some $k$ so that for all $i \geq k$,
$b_s(\rho(\omega^{[i]}))$ is infinite. If this is not the case, then property
(2) is already true for $\rho$ no matter how we define the remaining $c_s$ and so we may define $D_{s+1} = D_s$, and let
$c_{s+1}$ be an arbitrary computable injection so that $c_{s+1}(n) > n$ for all $n$, 
$\ran(c_{s+1})$ is disjoint from $\ran(c_0) \union
\ldots \union \ran(c_s) \union D_{s+1}$ and so that $\ran(c_0) \union
\ldots \union \ran(c_{s+1}) \union D_{s+1}$ is coinfinite.

So fix $k$ so that for all $i \geq k$, $b_s(\rho(\omega^{[i]}))$ is
infinite. Now we can find a computable set $D_{s+1} \supset D_s$ so that
for every $i \geq k$, $D_{s+1} \inters b_s(\rho(\omega^{[i]}))$ is infinite
and $D_{s+1}$ is disjoint from $\ran(c_0) \union \ldots \union \ran(c_s)$. This is just by the standard fact that if $(A_i)_{i \in \omega}$ is a uniformly c.e. family of infinite sets, then there is a computable infinite-coinfinite set $D$ so that $D \inters A_i$ is infinite for each $i$. More precisely, we will define $D_{s+1}$ so that $D_{s+1} \setminus D_s \subset \ran(b_s)$. At each step, 
define $D_{s+1} \setminus D_s$ on a large enough finite segment to ensure
that there are at least $n$ elements of $b_s(\rho(\omega^{[i]}))$ in
$D_{s+1}$ for every $k \leq i \leq n$. At step $n$ we also choose $n$ new elements not in
$D_{s} \union \ran(c_0) \union \ldots \union \ran(c_s)$ and promise that
they will not be in $D_{s+1}$ (so that at the end of the construction
$D_{s+1} \union \ran(c_0) \union \ldots \union \ran(c_s)$ is coinfinite). 

Once we have defined $D_{s+1}$ as above, no matter how we finish the construction, we will have 
the property that for every $n$ such that $b_s(\rho(n)) \in D_{s+1}$, $\max d(\rho(n))
\leq s$. This is since $d(\rho(n)) =
d_s(\rho(n)) \concat d(b_s(\rho(n))) = d_s(\rho(n))$ since
$b_s(\rho(n))$ is not in the range of any $c_m$ since it is in
$D_{s+1}$. By definition of $d_s$, we have $\max d_s(m) \leq s$ for all
$m$. Finally, the set $B = \{n \colon b_s(\rho(n)) \in D_{s+1}\}$ is
computable (since $D_{s+1}$ and $b_s$ are computable) and is our desired
computable set. To finish, we define $c_{s+1}$ to 
be an arbitrary computable injection so that $c_{s+1}(n) > n$ for all $n$, 
$\ran(c_{s+1})$ is disjoint from $\ran(c_0) \union
\ldots \union \ran(c_s) \union D_{s+1}$ and so that $\ran(c_0) \union
\ldots \union \ran(c_{s+1}) \union D_{s+1}$ is coinfinite.
\end{proof}

Of course, the range $\rho(\omega)$ of a computable function $\rho \from
\omega \to \omega$ is just a c.e.\ set, and we could equivalently state
Lemma~\ref{lem:coding_lemma} to be about c.e.\ sets instead. 
For example, 
part (2) of Lemma~\ref{lem:coding_lemma} would become:
if $(A_i)_{i \in \omega}$ is a
uniformly c.e. family of subsets of $\omega$, then either for infinitely
many $i$, $b_s(A_i)$ is
finite, or there is a computable set $C \subset \bigunion_i A_i$ so that $\max d(n) \leq s$ for all
$n \in C$ and for all but finitely many $i$, $b_s(A_i \inters C)$ is
infinite. Here $(A_i)_{i \in \omega}$ is
$(\rho(\omega^{[i]}))_{i \in \omega}$, and $B$ in the above lemma would be
$\rho^{-1}(C)$. We stated the Lemma~\ref{lem:coding_lemma} in the above form since this is
the way it will eventually be used, where $\rho$ is some many-one reduction.

Two important ideals in the proof of Theorem~\ref{m_construction}
will be the first and second iterated Fr\'echet
ideals on $\omega$ which we denote $I_1$ and $I_2$.  
We use $I_2$ when we are simultaneously analyzing all the
columns of a many-one reduction. 
\begin{defn}
Let $I_1 = \{A \subset \omega \colon A \text{ is finite}\}$.
Let $I_2 = \{A \subset \omega \colon$ for all but finitely many $i$,
$A^{[i]}$ is finite$\}$. 
\end{defn}

An important idea in our proof of Theorem~\ref{m_construction} is captured
by the following simple proposition. 
One should think here of a set not being in an ideal $I$ on $\omega$
as a notion of largeness. For example for the
Fr\'echet ideal $I_1$, $A \notin I_1$ if and only if $A$ is infinite.
\begin{prop}\label{prop:arithmetic_tree}
  Suppose $S \subset \omega^{< \omega}$ is a finitely branching tree, $t
  \from \omega \to S$ is an arithmetic function, and $I$ is an
  arithmetically definable ideal on $\omega$ (such as $I_1$ or $I_2$). Let $T \subset S$
  be defined by $T = \{s \in S \colon \{n \colon t(n) \supset
  s\} \notin I\}$. Then $T$ is an arithmetically definable subtree of $S$. Furthermore any $s \in T$ with no extensions in $T$ has $\{n
  \colon t(n) = s\} \notin I$. So by K{\H o}nig's lemma, either $T$ has an infinite branch
  and hence an arithmetically definable infinite branch, or there is some
  $s$ so that $\{n \colon t(n) = s\} \notin I$.
\end{prop}
\begin{proof}
  First we show $T$ is closed downward and is hence a tree. Suppose $s_1
  \in T$, and $s_0 \subset s_1$. Then since $\{n \colon t(n) \supset s_0\}
  \supset \{n \colon t(n) \supset s_1\}$ and any superset of a set not in
  $I$ is also not in $I$, we have $s_0 \in T$. 

  Now if $s \in T$, and $s_0, \ldots, s_k$ are the immediate extensions of
  $s$ in $S$, then we can partition the set $\{n \colon t(n) \supset s\}$ which is not in $I$ into
  finitely many sets: $\{n \colon t(n) = s\}$, and $\{n \colon t(n) \supset
  s_i\}$ for each $i \leq k$. 
  At least one of these sets must not be
  in $I$ since a union of finitely many sets in $I$ is in $I$. Hence,
  any $s \in T$ with no extensions in $T$ has $\{n \colon t(n) = s\} \notin
  I$.
\end{proof}

In the proof of Theorem~\ref{m_construction} we will 
use the same idea as the above proposition, but in a relativized
form, and where $t$ is a function to a finitely branching tree in a different
space (a tree made of elements of $[x]_E^{<\omega}$).

We are ready to prove our main theorem showing that a positive answer to
Question~\ref{hyp-rec-fin} implies the existence of nonuniform invariant
functions that are incomparable with the identity function.

\begin{thm}\label{m_construction}
  Suppose $E$ is a hyper-Borel-finite Borel equivalence relation on
  $\cantor$, and $(f_i)_{i \in \omega}$ is a countable collection of Borel functions from $\cantor$ to $\cantor$, so $f_i(x)$ is incomputable for every $i$ and $x$.  
  Then there exists an injective
  Borel function $f \from \cantor \to \cantor$ such that for all $x_0, x_1
  \in \cantor$ 
  \begin{enumerate}
  \item If $x_0 \mathrel{E} x_1$, then $f(x_0) \equiv_1 f(x_1)$
  \item If $x_0 \mathrel{\cancel{E}} x_1$, then $f(x_0) \not \equiv_m f(x_1)$.
  \item For every $x \in \cantor$ and $i \in \omega$, $f(x)$ is $\leq_m$-incomparable with $f_i(x)$. 
  \item For all $x \in 2^\omega$, there does not exist an infinite sequence
  $(x_i)_{i \in \omega}$ of distinct reals such that
  $\bigoplus_i f(x_i) \leq_m f(x)$.
  \end{enumerate}
\end{thm}
\begin{proof}%[Proof of Theorem~\ref{m_construction}]
  Let $\F_\omega \actson 2^\omega$ be a Borel action of the group
  $\F_\omega$ that generates the equivalence relation $E$. Let
  $(\gamma_i)_{i \in \omega}$ be a computable enumeration of the group
  $\F_\omega$ so that group multiplication is computable.
  Let $h_i \from \cantor \to (\cantor)^\omega$ be the Borel function where
  $h_i(x) \in (\cantor)^{\omega}$ is the $i$th real arithmetically
  definable from $\bigoplus_{j \in \omega} \gamma_j \cdot x$ (using some
  computable bijection to identify 
  $\cantor$ with $(\cantor)^\omega$). Intuitively, 
  $\bigoplus_{j \in \omega} \gamma_j \cdot x$ codes the
  entire orbit of $x$ under the group action. For example, for every
  $x \in \cantor$, the
  stabilizer of $x$ (i.e. $\{i \colon \gamma_i \cdot x = x\}$) is
  arithmetically definable from $\bigoplus_{j \in \omega} \gamma_j \cdot x$.
  Since the
  function $x \mapsto \bigoplus_{j \in \omega} \gamma_j \cdot x$ is Borel,
  each $h_i$ is Borel since it is the composition of a Borel function with an
  arithmetic function. Let $(E_j)_{j \in \omega}$ be a witness
  that $E$ is hyper-$(h_i)_{i \in \omega}$-finite, so $E_0 \subset E_1
  \subset \ldots$, and $E = \bigunion_{j \in
  \omega} E_j$.

  Let $g \from \cantor \to
  \cantor$ be the function $f$ from Lemma~\ref{genericity_lemma},  
  so the range of $g$ is a set of mutual $1$-generic reals, and if $x \mathrel{E} y$ and $z \leq_m g(x)$ via a many-one reduction with infinite range, then $z \nleq_m f_i(y)$ for all $i \in \omega$.
  Let $(c_m)_{m \in \omega}$ be as in Lemma~\ref{lem:coding_lemma}.  
  We define $f \from \cantor \to \cantor$ by:
  \[f(x)(n) = \begin{cases} 
  f(\gamma_i \cdot x)(c_{\<i,j\>}^{-1}(n)) & \text{ if $\exists i, j$ so $n
  \in \ran(c_{\<i,j\>})$ and $x \mathrel{E_j} \gamma_i \cdot x$} \\
  g(x)(n) \text{ otherwise}.
  \end{cases} \]
  This definition is self-referential, but it is not circular. If $f(x)(n)
  = f(\gamma_{i_0} \cdot x)(n_0)$ where $n_0 = c^{-1}_{\<i_0,j_0\>}(n)$,
  then $n_0 < n$
  since $c_{\<i_0,j_0\>}(n) > n$ for all $n$ by Lemma~\ref{lem:coding_lemma}. 
  So after finitely many applications of
  the definition of $f$ we will reach the base case of the definition and
  find a sequence $i_0, \ldots, i_k$ and $n_k$ where
  $f(x)(n) = g(\gamma_{i_k} \cdots \gamma_{i_0} \cdot x)(n_k)$. These kinds
  of self-referential definitions where we code values of $f$ into itself have been used
  before in the study of the Borel complexity of equivalence relations
  from computability theory (see e.g. \cite[Theorem 2.5]{MSS} and
  \cite[Theorem 3.6]{M}). 

  By the definition of $f$, part (1) of the theorem is true. Given any
  $x_0 \mathrel{E} x_1$, let $i$ be such that $\gamma_i \cdot x_0 = x_1$.
  There is some $j$ such that $x_0 \mathrel{E_j} x_1$. Then
  the function $c_{\<i,j\>}$ is a one-one reduction witnessing $f(x_1)
  \leq_1 f(x_0)$. This is because for all $n_0 \in \omega$, $f(\gamma_i \cdot
  x_0)(n_0) = f(x_0)(c_{\<i,j\>}(n_0))$ by letting $n = c_{\<i,j\>}(n_0)$ in the
  definition of $f(x_0)$. Arguing symmetrically, we also have $f(x_0) \leq_1
  f(x_1)$. The function $f$ is injective since $f(x) \restriction D_0 = g(x) \restriction D_0$ and since the elements of $\ran(g)$ are mutually $1$-generic, we cannot have $g(x) \restriction D_0 = g(y) \restriction D_0$ for distinct $x, y \in \cantor$.

  The idea of the proof is that $f$ is as generic as possible, subject to the coding we must do to ensure that if $x_0 \mathrel{E} x_1$, then $f(x_0)
  \equiv_1 f(x_1)$. Intuitively, there are two types of bits $n$ of $f(x)$.
  There are infinitely many ``generic'' bits $n$ where $f(x)(n) = g(x)(n)$. The
  remaining bits are used for coding where we record the value of the bits of
  $f(\gamma_i \cdot x)$ for $i \in \omega$. This coding scheme is also chosen to
  be generic (as made precise by Lemma~\ref{lem:coding_lemma}), and we use 
  our hyper-$(h_i)_{i \in \omega}$-finiteness witness to choose at what stage of the
  construction to code $f(y)$ into $f(x)$ for $y$ such that $y \mathrel{E} x$. 
  Supposing $z \leq_m f(x)$, the crux of the proof is understanding how well
  this many-one reduction can iteratively decode this coding to find bits of
  $g(\gamma_i \cdot x)(n)$ for many different $i$ and $n$. The high-level idea
  is that if there is a many-one reduction whose range decodes to
  be values of $g(\gamma_i \cdot x)(n)$ for a ``large'' set of $i$ and $n$
  (according to some ideal), then we get a contradiction to $(E_j)_{j \in
  \omega}$ being a hyper-$(h_i)_{i \in \omega}$-finiteness witness for $E$. But
  if a many-one reduction only uses values of $g(\gamma_i \cdot x)(n)$ for
  finitely many $i$ on a large set, then since the reals $g(\gamma_i \cdot x)$
  are mutually $1$-generic (or have the stronger genericity properties given in Lemma~\ref{genericity_lemma}), $z$ cannot be anything ``interesting'' and we show $f$ has properties (2), (3), and (4). 

  Our next goal is to give a definition of $f(x)$ 
  that is only in terms of the function $g$ and is not 
  self-referential. First we make a definition that describes when
  we
  recursively use the first clause $f(x)(n) = f(\gamma_i \cdot
  x)(c^{-1}_{\<i,j\>}(n))$ of the definition of $f(x)(n)$ to ``decode'' it.
  Say a sequence $(\<i_0, j_0\>, \<i_1, j_1\>, \ldots,
  \<i_k, j_k\>) \in \omega^{< \omega}$ is \define{$x$-valid} if 
  \[(\gamma_{i_{m-1}}
  \cdots \gamma_{i_0} \cdot x) \mathrel{E_{j_m}} (\gamma_{i_{m}} \cdots
  \gamma_{i_0} \cdot x)\] 
  for every $m \leq k$.
  Note that if a sequence is $x$-valid then every initial segment of it is
  $x$-valid.

  Let $d_x(n)$ be the longest initial segment of $d(n)$ that is $x$-valid.
  So $d_x \colon \omega \to \omega^{< \omega}$.
  Hence if $d_{x}(n) = (\<i_0, j_0\>, \<i_1, j_1\>, \ldots, \<i_k, j_k\>)$,
  then 
  \[f(x)(n) = f(\gamma_{i_0} \cdot x)(c^{-1}_{\<i_0,j_0\>}(n))\]
  by the definition of $f$ since $n \in
  \ran(c_{\<i_0,j_0\>})$ by the definition of $d$, and since $x
  \mathrel{E_{j_0}} \gamma_{i_0} \cdot x$ by the definition of being
  $x$-valid. Similarly, we have inductively that for every $m \leq k$,
  \[
  \begin{split}
  f(\gamma_{i_{m-1}} \cdots \gamma_{i_0} \cdot x)(c^{-1}_{\<i_{m-1},j_{m-1}\>} \circ
  \ldots \circ c^{-1}_{\<i_0,j_0\>}(n)) \\
  = f(\gamma_{i_{m}} \cdots
  \gamma_{i_0} \cdot x)(c^{-1}_{\<i_{m},j_{m}\>} \circ
  \ldots \circ c^{-1}_{\<i_0,j_0\>}(n))
  \end{split}
  \]
  again using the definition of $f$, the definition of $d$, and since 
  $(\gamma_{i_{m-1}}
  \cdots \gamma_{i_0} \cdot x) \mathrel{E_{j_m}} (\gamma_{i_{m}} \cdots
  \gamma_{i_0} \cdot x)$ by the definition of being $x$-valid. Finally,
  either $d_x(n) = d(n)$ and so 
  $c^{-1}_{\<i_{k},j_{k}\>} \circ \ldots \circ
  c^{-1}_{\<i_0,j_0\>}(n) \notin \ran(c_m)$ for any $m$ by the definition
  of $d$, or $d_x(n)$ is a proper initial segment of $d(n) = (\<i_0, j_0\>, \<i_1, j_1\>, \ldots, \<i_k, j_k\>,
  \<i_{k+1}, j_{k+1}\>, \ldots)$, so  
  $c^{-1}_{\<i_{k},j_{k}\>} \circ \ldots \circ
  c^{-1}_{\<i_0,j_0\>}(n) \in \ran(c_{\<i_{k+1},j_{k+1}\>})$
  but 
  $(\gamma_{i_{k}}
  \cdots \gamma_{i_0} \cdot x) \mathrel{\cancel{E_{j_{k+1}}}} (\gamma_{i_{k+1}} \cdots
  \gamma_{i_0} \cdot x)$, since $d_x(n)$ is the longest initial segment of
  $d(n)$ that is $x$-valid.
  In either case, in the definition of 
  $f(\gamma_{i_{k}} \cdots \gamma_{i_0} \cdot x)(c^{-1}_{\<i_{k},j_{k}\>} \circ
  \ldots \circ c^{-1}_{\<i_0,j_0\>}(n))$ we use the second clause of the
  definition, and so 
  \[
  \begin{split}
  f(\gamma_{i_{k}} \cdots \gamma_{i_0} \cdot x)(c^{-1}_{\<i_{k},j_{k}\>} \circ
  \ldots \circ c^{-1}_{\<i_0,j_0\>}(n)) \\
  = g(\gamma_{i_{k}} \cdots
  \gamma_{i_0} \cdot x)(c^{-1}_{\<i_k,j_k\>} \circ \ldots \circ
  c^{-1}_{\<i_0,j_0\>}(n)).
  \end{split}\]
  Putting together the above three displayed equations, we have shown that
  if $d_x(n) = (\<i_0, j_0\>, \<i_1, j_1\>, \ldots, \<i_k, j_k\>)$, then 
  we have the following explicit definition of $f(x)$ in terms of $g$.
  \[f(x)(n) = g(\gamma_{i_{k}} \cdots
  \gamma_{i_0} \cdot x)(c^{-1}_{\<i_k,j_k\>} \circ \ldots \circ
  c^{-1}_{\<i_0,j_0\>}(n)).\]
  To make this definition more compact, we introduce two more functions. 
  Define $y_x \from \omega \to [x]_E$ and $b_x \from \omega \to \omega$ as
  follows. If $d_x(n) = (\<i_0, j_0\>, \ldots,
  \<i_k,j_k\>)$, then $y_x(n) = \gamma_{i_{k}} \cdots
  \gamma_{i_0} \cdot x$ and $b_x(n) = c^{-1}_{\<i_k,j_k\>} \circ \ldots \circ
  c^{-1}_{\<i_0,j_0\>}(n)$. Hence for all $n$,
  \[f(x)(n) = g(y_x(n))(b_x(n)). \tag{*}\]
  That is for all $n$, $f(x)(n)$ codes the bit $b_x(n)$ of $g(y_x(n))$.
  Note that for all $n$, $b_x(n) \geq b(n)$ since $d_x(n)$ is an initial
  segment of $d(n)$. 

  Similarly, we define $d_{s,x} \from \omega \to \omega^{< \omega}$ by
  letting $d_{s,x}(n)$ be the longest initial segment of $d_s(n)$ that is
  $x$-valid. Note that $d_x(n) = d_{s,x}(n)$ for sufficiently large $s$
  (i.e. $s \geq \max d(n)$). Define also $y_{s,x} \from \omega \to [x]_E$
  and $b_{s,x} \from \omega \to \omega$ as follows. If $d_{s,x}(n) =
  (\<i_0, j_0\>, \ldots, \<i_k,j_k\>)$, then $y_{s,x}(n) = \gamma_{i_{k}}
  \cdots \gamma_{i_0} \cdot x$ and $b_{s,x}(n) = c^{-1}_{\<i_k,j_k\>} \circ
  \ldots \circ c^{-1}_{\<i_0,j_0\>}(n)$. An identical kind of induction to
  the one above using the properties of being $x$-valid shows that for all
  $n$ and $s$, \[f(x)(n) = f(y_{s,x}(n))(b_{s,x}(n)). \tag{**}\] Note,
  though, that in this equation (**) we have $f$ on the right hand side
  instead of $g$. This is because it is possible that $d_x(n) \supsetneq
  d_s(n)$ and so $n$ needs to be further decoded using functions $c_m$ for
  $m > s$.

  Our analysis of reals that are many-one reducible to $f(x)$ will be based
  on analyzing a finitely branching tree built out of elements in $[x]_E$,
  which is related to (*) above. But it does not reflect the above equations exactly, because we will need these trees (and the functions $t_x$ and $t_{s,x}$) to be arithmetically definable relative to $\bigoplus_{j \in \omega} \gamma_j \cdot x$, and for them to not depend on the witness $(E_j)_{j \in \omega}$. Let $[x]_E^{< \omega}$ be the set of
  finite sequences $(y_0, \ldots, y_l)$ so that $y_i \in [x]_E$ for all $i
  \leq l$. We define a function $t_{x}(n) \from \omega \to [x]_E^{<
  \omega}$ as follows. Given $d(n) = (\<i_0,j_0\>, \ldots, \<i_k,j_k\>)$,
  consider the sequence $(x, \gamma_{i_0} \cdot x, \ldots, \gamma_{i_k}
  \cdots \gamma_{i_0} \cdot x)$. 
  This sequence may contain elements that
  are repeated so we define $t_{x}(n)$ to be a ``de-duplicated'' version
  of this sequence, so $t_{x}(n) = (y_0, \ldots, y_l)$ has the same
  elements as $(x, \gamma_{i_0} \cdot x, \ldots, \gamma_{i_k} \cdots
  \gamma_{i_0} \cdot x)$, but where each element occurs exactly once.
  Precisely, let $y_0 = x$ and $y_{j+1}$ be the first element of the
  sequence $(x, \gamma_{i_0} \cdot x, \ldots, \gamma_{i_k} \cdots
  \gamma_{i_0} \cdot x)$ that is not equal to $y_m$ for any $m \leq j$.
  Intuitively, if $t_{x}(n) = (y_0, \ldots, y_l)$, this means $y_0 = x$, and presuming that $d(n)$ is $x$-valid,
  then $f(y_0)(n)$ codes a bit of $f(y_1)$ which codes a bit of $f(y_2)$,
  \ldots, which codes a bit of $f(y_l)$, which is equal to a bit of
  $g(y_i)$ for some $i \leq l$. (We ultimately code a bit of $g(y_i)$ for some $i \leq l$ instead of a bit of $g(y_l)$ because of how we have de-duplicated this sequence, and so $y_x(n)$ may not equal the last element $y_l$ of the sequence $t_x(n)$, even if $d(n)$ is $x$-valid).   
  Note that regardless of whether $d(n)$ is
  $x$-valid, $y_x(n)$ is an element of $t_x(n)$. 
  
  One final fact we
  will often use about the relationship between $y_x(n)$ and $t_x(n)$ is
  that if $\max d(n) \leq s$, $r = (y_0, \ldots, y_l)$, $y_l
  \mathrel{\cancel{E_s}} x$, and $t_x(n) \supset r$, then $y_x(n) = y_i$
  for some $i \leq l$. That is, in this case even though $t_x(n)$
  may contain many elements not in $r$, the value $y_x(n)$ must come from
  $r$. This is since any part of the sequence $d(n)$ that
  yields part of $t_x(n)$ that
  extends $r$ cannot be $x$-valid since $y_l
  \mathrel{\cancel{E_s}} x$, and $\max d(n) \leq s$.

  Note that we are defining $t_x(n)$ using the function
  $d(n)$ instead of $d_x(n)$ because we want $t_x(n)$ to be arithmetically
  definable relative to $\bigoplus_{j \in \omega} \gamma_j \cdot x$. This
  is so we can use 
  the idea of Proposition~\ref{prop:arithmetic_tree}, relative to
  $\bigoplus_{j \in \omega} \gamma_j \cdot x$. 
  (The definition of $d_x(n)$ depends on our hyper-Borel-finiteness
  witness $(E_j)_{j \in \omega}$ and we have no bound on its complexity in the Borel hierarchy).

  We will also define a similar function to $t_x$ but using the function $d_s(n)$
  instead of $d(n)$. Precisely, 
  define $t_{s,x}(n) \from \omega \to [x]_E^{< \omega}$ as
  follows. Given $d_s(n) = (\<i_0,j_0\>, \ldots, \<i_k,j_k\>)$, let
  $t_{s,x}(n)$ be the de-duplicated version of the sequence $(x, \gamma_{i_0}
  \cdot x, \ldots, \gamma_{i_k} \cdots \gamma_{i_0} \cdot x)$ as in the
  definition of $t_x$. Note that
  since $d_s(n) = d(n)$ for $s \geq \max d(n)$, we have that $t_x(n) =
  t_{s,x}(n)$ if $s \geq \max d(n)$. An important property of 
  $t_{s,x}$ is that its values (unlike $t_x$) form a
  finitely branching tree. Precisely, If $t_{s,x}(n) = (y_0, \ldots, y_l)$,
  we must have that for every $k \leq l$, $y_k = \gamma_i \cdot y_j$ for
  some $i \leq s$ and $j \leq k$. This is by definition of $d_s$ and $t_{s,x}$.
  Hence, the downward closure of all the values of $t_{s,x}(n)$ forms a
  finitely branching tree in $[x]_E^{< \omega}$. Mostly (except at the end
  of Claim 3), using Lemma~\ref{lem:coding_lemma} we will work on sets $B
  \subset \omega$ where $\max d(\rho(n)) \leq s$, and hence $t_{x}(\rho(n))
  = t_{s,x}(\rho(n))$ for all $n \in B$.

  Because we have introduced many different functions, we briefly
  summarize:
  \begin{itemize}
    \item $g \from \cantor \to \cantor$ is the generic function from
    Lemma~\ref{genericity_lemma}
    whose range is a
    set of mutual $1$-generics, and so that if $x \mathrel{E} y$ and $z
    \leq_m g(x)$ via a
    many-one reduction with infinite range, then $z \nleq_m f_i(y)$ for all $i
    \in \omega$. 
    \item $(E_j)_{j \in \omega}$ are the witness that $E$ is
    hyper-$(h_i)_{i \in \omega}$-finite. The functions $(h_i)_{i \in
    \omega}$ are those that are arithmetically definable from
    $\bigoplus_j \gamma_j \cdot x$ (i.e. arithmetically definable from the
    orbit of $x$). 
    \item $f \from \cantor \to \cantor$ is the Borel reduction from $E$ to
    $\equiv_m$ we're building. The definition of $f$ in terms of $g$, $(E_j)_{j
    \in \omega}$ and $(c_m)_{m \in \omega}$ is given at the beginning of
    the proof.
    \item $(c_m)_{m \in \omega}$ are the ``coding functions'' used to
    ensure that if $x \mathrel{E} y$, then $f(x) \leq_1 f(y)$. Precisely,
    if $x \mathrel{E_j} \gamma_i \cdot x$, then
    $f(\gamma_i \cdot x) \leq_1 f(x)$ via $c_{\<i,j\>}$. Each $c_m$ is
    computable, injective, and increasing, but the sequence $(c_m)_{m \in
    \omega}$ is not uniformly computable. The $c_m$ have disjoint ranges.
    The sequence $(c_m)_{m \in \omega}$ is a ``generic'' such sequence and is constructed
    in Lemma~\ref{lem:coding_lemma}. The functions $d, d_s \colon \omega \to
    \omega^{< \omega}$ and $b, b_s \from \omega \to
    \omega$
    are associated functions used for decoding and defined in
    Definition~\ref{defn:coding_decoding}.
    \item The function $d_x \from \omega \to \omega^{< \omega}$ is defined
    so that $d_x(n)$ is the longest initial segment of $d(n)$ that is $x$-valid,
    where we define $x$-valid sequences according to which clause of the
    definition of $f(x)(n)$ would be used to decode them. Using $d_x$, we
    then gave a definition (*) above of the function $f(x)$ just
    in terms of $g$: $f(x)(n) = g(y_x(n))(b_x(n))$, where $y_x \from \omega
    \to [x]_E$, and $b_x \from \omega \to \omega$ were defined in terms of
    $d_x(n)$. Similarly, $y_{s,x}$,
    $b_{s,x}$ and $d_{s,x}$ are defined analogously to $y_x$, $b_x$, and
    $d_x$ but using $d_s$ instead of $d$.
    Typically below (except at the end
    of Claim 3) we will work on
    sets $B \subset \omega$ on which $\max d(n) \leq s$, and hence there is
    no difference in these functions e.g. $y_x(n) = y_{s,x}(n)$, $b_{x}(n) = b_{s,x}(n)$, and $d_x(n) = d_{s,x}(n)$, $t_{x}(n) = t_{s,x}(n)$, for all $n \in B$.
    %Note that $f$, the definition of a sequence being $x$-valid, $d_x$, $y_x$, and $b_x$, and $y_{s,x}, b_{s,x}$, and $d_{s,x}$ all depend on the witness $(E_j)_{j \in \omega}$.

    \item The function $t_x \from \omega \to [x]_E^{< \omega}$ maps each
    bit $n$ to the sequence of distinct $y_0, y_1, \ldots, y_k$ where $y_0
    = x$ and 
    $f(x)(n)$ is a coded bit of $f(y_1)$ which is a coded bit of $f(y_2)$
    \ldots which is a coded bit of $f(y_l)$, assuming $d(n)$ is
    $x$-valid. Note that $y_{x}(n)$ is an element of $t_{x}(n)$ for all
    $n$. The function $t_{s,x}$ is defined similarly to $t_x$, except where
    we use the sequence $d_s(n)$ instead of $d(n)$. We use this function
    $t_{s,x}$ because its values
    form a finitely branching tree. Similarly to above, we will typically work on sets $B \subset \omega$ on which $\max d(n) \leq s$, and hence $t_{x}(n) = t_{s,x}(n)$ for all $n \in B$.

    \item We emphasize that $t_x$ and $t_{s,x}$ are arithmetically definable from $\bigoplus_{j \in \omega} \gamma_j \cdot x$, and they do not depend on the witness $(E_j)_{j \in \omega}$. 
    \item The functions $f(x)$, $d_x, y_x, b_x, d_{s,x}, y_{s,x}, b_{s,x}$ all depend on the witness $(E_j)_{j \in \omega}$.
  \end{itemize}

  Now $d_x$ and $b_x$ are not computable in general since $d$ is not
  computable and the set of $x$-valid sequences is also not computable in
  general. However, there are certain computable infinite subsets of $\omega$ on which 
  $d_{s,x}$ and $b_x$ \emph{are} computable. This is key to our arguments:

  \begin{claim1}
    Suppose $\rho \colon \omega \to \omega$ is
    computable, $r = (y_0, \ldots, y_l) \in [x]_E^{< \omega}$, $y_i$ is an
    element of $r$, and $s \in \omega$. Then
    \begin{enumerate} 
    \item $A = \{n \in \omega \colon
    t_{s,x}(n) = r \land y_{s,x}(n) = y_i\}$ is computable, and $d_{s,x }
    \restriction A$ is computable. Hence if $B \subset \omega$ is
    computable and $\max d(\rho(n)) \leq s$ for all $n \in B$, then 
    $A' = \{n \in B \colon
    t_x(\rho(n)) = r \land y_x(\rho(n)) = y_i\}$ is computable and $b_x
    \circ \rho$ is computable on $A'$. 

    \item If $y_l \mathrel{\cancel{E_s}} x$, then 
    $A = \{n \colon t_{s,x}(\rho(n)) \supset
    r \land y_x(\rho(n)) = y_i\}$ and $d_{s,x} \restriction A$ are
    computable. 
    Hence if $B \subset \omega$ is
    computable and $\max d(\rho(n)) \leq s$ for all $n \in B$, then 
    $A' = \{n \in B \colon
    t_x(\rho(n)) \supset r \land y_x(\rho(n)) = y_i\}$ is computable and $b_x
    \circ \rho$ is computable on $A'$. 
\end{enumerate}
  \end{claim1}
  \begin{proof}\renewcommand{\qedsymbol}{$\square$ Claim 1.} 
    The idea is that given $r$ and $s$, there is a finite amount of
    information about how group elements $\gamma_i$ for $i \leq s$ act
    between elements of $r$, and how elements of $r$ are $E_j$ related for
    $j \leq s$. From this we can compute all of the above.

    More precisely, the set of tuples $(i,j_0,j_1)$ such that $i \leq s$ and $j_0,j_1 \leq
    l$ and $\gamma_i \cdot y_{j_0} = y_{j_1}$ is finite. Suppose we are
    given $d_s(n) = (\<i_0, j_0\>, \ldots,
    \<i_k,j_k\>)$, where $i_m, j_m \leq s$ for every $m \leq k$ by
    definition of $d_s$. Then for each $m \leq k$ we can iteratively compute which
    element of $r$ is equal to $\gamma_{i_{m}} \cdots \gamma_{i_0} \cdot
    x$, provided all previous values of $\gamma_{i_{m'}} \cdots
    \gamma_{i_0} \cdot x$ for $m' < m$ have been elements of $r$. We can
    also similarly compute the least $m$ so that $\gamma_{i_{m}} \cdots
    \gamma_{i_0} \cdot x$ is not an element of $r$. 

    Similarly, the set of tuples $(i_0, i_1, j)$ such that $j \leq s$ and
    $i_0, i_1 \leq l$ so that $y_{i_0} \mathrel{E_j} y_{i_1}$ is finite.
    From this information, if $t_{s,x}(n) = r$, we can 
    determine what subsequences of $d_{s}(n)$ are $x$-valid, and hence
    compute $d_{s,x}(n) \restriction A$ in case (1). In case (2), note that
    since $y_l \mathrel{\cancel{E_s}} x$, the least $m$ so that 
    $(\<i_0,j_0\>, \ldots, \<i_m,j_m\>)$ is not $x$-valid must have the
    property that 
    $\gamma_{i_{m'}} \cdots \gamma_{i_0} \cdot x$ is an element of $r$ for
    all $m' \leq m$. Hence in this case we can also compute $d_{s,x}
    \restriction A$. The claim follows. 
  \end{proof}

  We will prove two main claims about $z \in \cantor$ such that $z \leq_m
  f(x)$. Recall that if $y, z \in \cantor$ and $A
  \subset \omega$ is computable, by $z \restriction A \leq_m y$ we mean
  there is a computable function $\rho \from A \to \omega$ so that for all
  $n \in A$, $z(n) = y(\rho(n))$.

  \begin{claim2}
    Suppose $x,z \in \cantor$ are such that $z \leq_m f(x)$, and $z$ is
    incomputable. Then there is a computable infinite set $A \subset
    \omega$ and some $y \mathrel{E} x$ so that $z \restriction A \leq_m
    g(y)$ via a many-one reduction with infinite range. 
  \end{claim2}
  \begin{proof}\renewcommand{\qedsymbol}{$\square$ Claim 2.} 
  Let $\rho \from \omega \to \omega$ be the many-one reduction witnessing
  $z \leq_m f(x)$. 
  The idea of the proof is to make a finitely branching tree $T$ of
  elements of $[x]_E$ where $(y_0, \ldots, y_l) \in T$ means that
  a ``large'' (according to some ideal) number of bits $f(x)(\rho(n))$ code values
  of $f(y_0)$ which code values of $f(y_1)$ \ldots which code values of
  $f(y_l)$ (assuming the code is $x$-valid). If the tree is finite, a ``large'' number of bits of the
  many-one reduction can be many-one reduced to a single $g(y)$ for $y \in
  [x]_E$. If the tree is infinite, some finite branch $r$ in the tree must be
  coded in a way that is not $x$-valid, otherwise we would
  contradict that hyper-$(h_i)$-finiteness of the $(E_j)_{j \in \omega}$
  (since our tree will be arithmetically definable relative to $\bigoplus_{i
  \in \omega} \gamma_i \cdot x$).
  Then we can find a ``large'' set of incorrectly coded bits corresponding
  to extensions of $r$ 
  that reduce to a single $g(y)$. We will make this tree using the
  same idea as Proposition~\ref{prop:arithmetic_tree} using the function
  $t_x$. 

  We break into two cases depending on which case holds
  for $\rho$ in Lemma~\ref{lem:coding_lemma}.(1). 

  Case 1: there is a computable set $B$ and an $s$ so that $\max d(\rho(n))
  \leq s$ for all $n \in B$ and $b_s(\rho(B))$ is infinite. 

  In this case, let $I$ be the ideal on subsets of $B$ where for $A \subset
  B$, we have $A \in I$ if $b_s(\rho(A))$ is finite. Let $T = \{r \in
  [x]_E^{< \omega} \colon \{n \in B \colon t_x(\rho(n)) \supset r\} \notin
  I\}$.  Hence $T$ is a finitely branching tree analogously to
  Proposition~\ref{prop:arithmetic_tree}, and it is arithmetically
  definable relative to $\bigoplus_{i \in \omega} \gamma_i \cdot x$. (The
  reason we are using this ideal $I$ rather than the Fr\'echet ideal $I_1$
  is in order to make the proof that the many-one reduction has infinite
  range easier).

  If $T$ is finite, as in Proposition~\ref{prop:arithmetic_tree}, there
  must be some $r \in T$ such that $\{n \in B \colon t_x(\rho(n)) = r\}
  \notin I$. Let $A = \{n \in B \colon t_x(\rho(n)) = r\}$. Let $r = (y_0,
  \ldots, y_l)$. Since $y_x(n)$ is an element of $t_x(n)$ for every $n$, we
  can partition $A$ into the finitely many sets $A_i = \{n \in B \colon
  t_x(\rho(n)) = r \land y_x(\rho(n)) = y_i\}$ for each $i \leq l$. Hence,
  there must be some $y_i$ so that the set $A_i \notin I$. Fix this $i$.
  Now for every $n \in A_i$, $f(x)(\rho(n)) = g(y_i)(b_x(\rho(n)))$ by (*).
  Since by Claim 1, $b_x \circ \rho$ is computable on $A_i$, we therefore
  have $z \restriction A_i \leq_m g(y_i)$. To see this many-one reduction
  has infinite range note first that $b_s(\rho(A_i))$ is infinite by
  definition of $I$, $b(\rho(n)) = b_s(\rho(n))$ for all $n \in A_i$ (since
  $\max d(\rho(n)) \leq s$ for all $n \in B$), and so $b(\rho(A_i))$ is
  infinite. Finally, $b_x(m) \geq b(m)$ for all $m$ by definition of $b_x$,
  and so $b_x(\rho(A_i))$ is infinite. 

  Now suppose $T$ is infinite. Then there is an infinite branch in $T$ that
  is arithmetically definable from $\bigoplus_{i \in \omega} \gamma_i \cdot
  x$. Since each $t_x(n)$ contains no repeated elements by definition, the
  set of $y \in [x]_E$ in this branch is infinite. So there is some
  $r = (y_0, \ldots, y_l) \in T$ in this branch so that $y_l
  \mathrel{\cancel{E_s}} x$. Otherwise, this would contradict that $E_s$ is
  $(h_i)$-finite. Now for all $n$ such that $t_x(n) \supset r$, we must
  have $y_x(n) = y_i$ for some $i \leq l$. So since $\{n \in B \colon
  t_x(\rho(n)) \supset r\} \notin I$, there must be some $y_i$ so that $A =
  \{n \in B \colon t_x(\rho(n)) \supset r \land y_x(\rho(n)) = y_i\} \notin I$.
  Since $f(x)(\rho(n)) = g(y_i)(b_x(\rho(n)))$ for all $n \in A$, we have
  $z \restriction A \leq_m g(y_i)$ by Claim 1 since $b_x$ is computable on
  $A$. This many-one reduction has infinite range on $A$ by the same
  argument as the above paragraph: $b_s(\rho(A))$ is infinite, $b_s(\rho(A))
  = b(\rho(A))$, and $b_x(m) \geq b(m)$ for all $m$.

  Case 2: There is an $s \in \omega$ so that $b_s(\rho(\omega))$ is finite. 

  Let $s'$ be larger than both $s$ and $\max d(b_s(\rho(n)))$ for all $n
  \in \omega$. This is finitely many values since there are only finitely
  many values of $b_s(\rho(n))$. So $\max d(\rho(n)) \leq s'$ for all $n
  \in \omega$ since $d(n) = d_s(n) \concat d(b_s(n))$ for every $n, s$.
  Let $T = \{r \in [x]_E^{< \omega} \colon \{n \colon
  t_x(\rho(n)) \supset r\} \text{ is infinite}\}$. $T$ is a finitely
  branching tree as in Proposition~\ref{prop:arithmetic_tree} since $\max
  d(\rho(n)) \leq s'$ for all $n \in \omega$, and so $t_{s',x}(\rho(n)) =
  t_x(\rho(n))$ for all $n \in \omega$. 

  If $T$ is finite, then for all but finitely many $n$, we have
  $t_x(\rho(n)) = r$ for some $r \in T$. For each $r = (y_0, \ldots, y_l)
  \in T$ and $i \leq l$, let $A_{r,i} = \{n \colon t_x(\rho(n)) = r
  \land y_x(\rho(n)) = y_i\}$. So all but finitely many $n \in \omega$ are
  in some $A_{r,i}$, and there are finitely many sets $A_{r,i}$. 
  By Claim 1, every $A_{r,i}$ is computable and $z \restriction A_{r,i}
  \leq_m g(y_i)$ for each $A_{r,i}$. If all these many-one reductions
  have finite range, then $z$ is computable, since there are finitely many
  $A_{r,i}$. This is a contradiction. So one of these many-one reductions
  $z \restriction A_{r,i} \leq_m g(y_i)$ has infinite range.

  Now suppose $T$ is infinite, and so there is an infinite
  branch in $T$ that is arithmetically definable from $\bigoplus_{i \in
  \omega} \gamma_i \cdot x$. The infinite set of $y \in [x]_E$ that appear
  in this branch is arithmetically definable from $\bigoplus_{i \in \omega}
  \gamma_i \cdot x$. So there is some $r = (y_0, \ldots, y_l)$ in this branch
  so that $y_l \mathrel{\cancel{E_s}} x$. Otherwise, this would contradict
  that $E_s$ is $(h_i)$-finite. Let $A = \{n \colon t_x(\rho(n)) \supset
  r\}$. Let $A_i = \{n \in A \colon y_x(\rho(n)) = y_i\}$, so $A_0, \ldots
  A_l$ partition $A$. Since $\{t_x(\rho(n)) \colon n \in A\}$ is infinite
  since it includes our infinite branch, we can find $i \leq l$ so that
  $\{t_x(\rho(n)) \colon n \in A_i\}$ is infinite. In particular, the lengths of
  these $|t_x(\rho(n))|$ where $n \in A_i$ are arbitrarily large. Then
  $A_i$ is computable and $b_x \circ \rho$ is computable on $A_i$ by Claim
  1. Since $f(x)(\rho(n)) = g(y_i)(b_x(\rho(n)))$ by (*) we have that 
  $z \restriction A_i \leq_m g(y_i)$. 

  We now show the many-one reduction $b_x \circ \rho$ witnessing $z
  \restriction A_i \leq_m g(y_i)$ has infinite range on $A_i$. For all $n
  \in A_i$, $t_x(\rho(n)) \supset r$, and the difference in their lengths
  is bounded by $|t_x(\rho(n))| - |r| \leq |d(\rho(n))| - |d_x(\rho(n))|$.
  This is because the elements of $t_x(\rho(n))$ that are not in $r$ must come
  from elements of $d(\rho(n))$ that are not $x$-valid (i.e. not in
  $d_x(\rho(n))$) since $y_l \mathrel{\cancel{E_s}} x$. Finally
  $|d(\rho(n))| - |d_x(\rho(n))| \leq b_x(\rho(n))$, since $b(\rho(n)) \geq 0$
  and $b(\rho(n))$ is obtained from $b_x(\rho(n))$ by taking additional
  inverse images of $b_x(\rho(n))$ by the elements of $c_m$ that are in
  $d(\rho(n))$ but not in $d_x(\rho(n))$, and the $c_m(n) > n$ for all $n$. 
  Hence, $b_x(\rho(n)) \geq |t_x(\rho(n))| - |r|$ and since the
  lengths of $|t_x(\rho(n))|$ are unbounded on $A_i$, the values of
  $b_x(\rho(n))$ are also unbounded on $A_i$.
  \end{proof}

  To show $f$ has property (2), we prove the contrapositive.
  Suppose $f(y) \leq_m f(x)$ for some $x, y \in \cantor$. By
  Lemma~\ref{lem:coding_lemma} there is a computable infinite set $D_0$
  so that $D_0$ is disjoint from $\ran(c_m)$ for every $m$, and hence $f(y)
  \restriction D_0 = g(y) \restriction D_0$, so $g(y) \restriction D_0 \leq_m
  f(x)$. Note that $g(y) \restriction D_0$ is incomputable, since any
  $1$-generic restricted to a computable set is incomputable.
  By Claim 2, there is some $y' \mathrel{E} x$ and
  infinite $A \subset D_0$ so that $g(y) \restriction A \leq_m g(y')$.
  (By applying the Claim to $z = \{n \colon \text{the $n$th element of $D_0$ is
  in $g(y)$}\}$. Note that $z \leq_m f(x)$.) 
  We
  must have $y = y'$, otherwise a computable subset of $g(y)$ is many-one
  reducible to $g(y')$ contradicting their mutual $1$-genericity. Hence $y
  = y' \mathrel{E} x$. Note we are only assuming here that $g$ maps to a set of mutual $1$-generics, and not that it has the stronger properties given in Lemma~\ref{genericity_lemma}.
  
  To show $f$ has property (3), for every $i \in \omega$, note first that $f(x)
  \nleq_m f_i(x)$. This is since on the infinite computable set $D_0$, from Lemma~\ref{lem:coding_lemma}
  $f(x) \restriction D_0 = g(x) \restriction D_0$, and so if $g(x) \restriction
  D_0 \leq_m f_i(x)$, this would contradict the properties of $g$ from
  Lemma~\ref{genericity_lemma}, letting $z = g(x) \restriction D_0$. Conversely, suppose $f_i(x) \leq_m f(x)$. Then by Claim 2, there is some $y \mathrel{E} x$ and infinite $A \subset D_0$ so that $f_i(x) \restriction A \leq_m g(y)$. But then $z = f_i(x) \restriction A$ contradicts the properties of $g$ from Lemma~\ref{genericity_lemma}.
  
  Now we prove a similar result to Claim 2 above but where we analyze all the
  columns of a many-one reduction using the ideal $I_2$. This is required to
  prove part (4) of the theorem.

  \begin{claim3}
  Suppose $x,z \in \cantor$ are such that $z \leq_m f(x)$, and
  $z^{[n]}$ is
  incomputable for every $n$. Then
  there is a computable set $B \subset \omega$ with $B \notin I_2$ and some
  $y \mathrel{E} x$ so that $z \restriction B \leq_m g(y)$. 
  \end{claim3}

  \begin{proof}\renewcommand{\qedsymbol}{$\square$ Claim 3.} 
  Let $\rho \from \omega \to \omega$ be the many-one reduction witnessing
  $z \leq_m f(x)$. We break into two cases depending on which case holds
  for $\rho$ in Lemma~\ref{lem:coding_lemma}.(2). 

  Case 1: There is a computable set $B$ so that $\max d (\rho(n))
  \leq s$ for all $n \in B$ and for all but finitely many $i$,
  $b_s(\rho(B^{[i]}))$ is infinite.

  In this case, we use a similar idea as in Claim 2. Note that $B \notin
  I_2$. Let
  $T = \{r \in
  [x]_E^{< \omega} \colon \{n \in B \colon t_x(\rho(n)) \supset r\} \notin
  I_2\}$. So as in Proposition~\ref{prop:arithmetic_tree}, $T$ is a
  finitely branching tree
  that is arithmetically definable relative to $\bigoplus_{i \in \omega}
  \gamma_i \cdot x$.

  Suppose $T$ is finite. Then as in
  Proposition~\ref{prop:arithmetic_tree}, there must be some $r \in T$ such
  that $\{n \in B \colon t_x(\rho(n)) = r\} \notin I_2$. Let $A = \{n \in
  B \colon t_x(\rho(n)) = r\}$. Let $r = (y_0, \ldots, y_l)$. We can
  partition $A$ into finitely many sets $A_i = \{n \in A \colon y_x(\rho(n))
  = y_i\}$ for each $i \leq l$, and so there must be some $y_i$ so that the
  set $A_i \notin I_2$. Now for every $n \in A_i$, $f(x)(\rho(n)) =
  g(y_i)(b_x(\rho(n)))$. Since by Claim 1, $A_i$ is computable and $b_x \circ \rho$ is computable on
  $A_i$, we therefore have $z \restriction A_i \leq_m g(y_i)$. 

  Now suppose $T$ is infinite, so there is an infinite branch in $T$ that
  is arithmetically definable from $\bigoplus_{i \in \omega} \gamma_i \cdot
  x$. There must be some $r = (y_0, \ldots, y_l)$ in this branch so that
  $y_l \mathrel{\cancel{E_s}} x$. Otherwise, this would contradict that
  $E_s$ is $(h_i)$-finite. Now for all $n$ such that $t_x(n) \supset r$, we
  must have $y_x(n) = y_i$ for some $i \leq l$. So since $\{n \in B \colon
  t_x(\rho(n)) \supset r\} \notin I_2$, there must be some $y_i$ so that $A
  = \{n \in B \colon t_x(\rho(n)) \supset r \land y_x(\rho(n)) = y_i\} \notin I_2$.
  $A$ is computable and $b_x \circ \rho \restriction A$ is computable by
  Claim 1. So since $f(x)(\rho(n)) = g(y_i)(b_x(\rho(n)))$ for all $n \in
  A$ we have $z \restriction A \leq_m g(y_i)$. 

  Case 2: There is an $s$ so that for infinitely many $i$,
  $b_s(\rho(\omega^{[i]}))$ is finite. Let $B = \bigunion \{\omega^{[i]}
  \colon b_s(\rho(\omega^{[i]})) \text{ is finite}\}$. $B$ is not
  necessarily computable, but it is arithmetical.
  %, and all of its vertical
  %sections $B^{[i]}$ are computable since they are all empty or equal to
  %$\omega^{[i]}$. 
  Now let $T = \{r \in [x]_E^{< \omega} \colon \{n
  \in B \colon t_{s,x}(\rho(n)) \supset r\} \notin I_2\}$. 

  If $T$ is infinite, then there must be some $r = (y_0, \ldots, y_l) \in
  T$ so that $y_l \mathrel{\cancel{E_s}} x$, otherwise there would be an
  infinite branch in $T$
  that is arithmetically definable
  from $\bigoplus_{i \in \omega} \gamma_i \cdot x$ and an infinite subset
  of $E_s$ contradicting that $(E_j)_{j \in \omega}$ is a
  hyper-$(h_i)$-finiteness witness. So fix an $r \in T$ so
  that $\{n \in B \colon t_{s,x}(\rho(n)) \supset r\} \notin I_2$. Then the
  larger set $A = \{n \in \omega \colon t_{s,x}(\rho(n)) \supset r\}$
  (where we have replaced $B$ with $\omega$) also has $A \notin I_2$.
  Finally, there must be some $y_i$ with $i \leq l$ so that $A_i = \{n \in
  A \colon y_x(\rho(n)) = y_i\}$ has $A_i \notin I_2$. This set $A_i$ is
  computable by Claim 1, and $z \restriction A_i \leq_m g(y_i)$. 

  If $T$ is finite and there is an $i$ so that $b_s(\rho(\omega^{[i]}))$ is
  finite and all but finitely many $n \in \omega^{[i]}$ have that $d_s(\rho(n))$
  is $x$-valid, then we claim $z^{[i]}$ is computable,
  which is a contradiction. Now $f(x)(n) = f(y_{s,x}(n))(b_{s,x}(n))$ for
  all $n$ by (**), and if $d_s(n)$ is $x$-valid, then $b_{s,x}(n) = b_s(n)$
  and $b_s$ is computable. So for each $r \in T$ and $y_j$ in $r$,
  $A_{r,j} = \{n \colon t_{s,x}(\rho(n)) = r \land y_{s,x}(\rho(n)) =
  y_j\}$ is computable by Claim 1, and $z \restriction (\omega^{[i]}
  \inters A_{r,j}) \leq_m f(y_j)$ via a reduction that has
  finite range since $b_s(\rho(\omega^{[i]}))$ is finite. So since the
  finitely many sets $A_{r,j} \inters \omega^{[i]}$ are computable and
  disjoint, and their union is equal to $\omega^{[i]}$ mod finite, we have
  that $z \restriction \omega^{[i]} = z^{[i]}$ is computable since we can
  partition it mod finite into finitely many computable pieces. 

  Thus, for all $i$ such that $\omega^{[i]} \subset B$, there are
  infinitely many $n \in \omega^{[i]}$ so $d_s(\rho(n))$ is not $x$-valid.
  So let $B' = \{n \in B \colon d_s(\rho(n)) \text{ is not $x$-valid}\}$.
  Then $B' \notin I_2$. Let $T' = \{r \in [x]_E^{< \omega} \colon \{n
  \in B' \colon t_{s,x}(\rho(n)) \supset r\} \notin I_2\}$. Then $T'$
  is finite since it is a subset of $T$, and there must be some $r \in
  T'$ and some $y_i \in r$ so that $\{n \in B' \colon
  t_{s,x}(\rho(n)) = r \land y_x(\rho(n)) = y_i\} \notin I_2$. Hence, the
  larger computable set: $A = \{n \colon t_{s,x}(\rho(n)) = r \land
  d_s(\rho(n)) \text{ is not $x$-valid} \land y_x(\rho(n)) = y_i\} \notin
  I_2$. Finally, $z \restriction A \leq_m g(y_i)$ by Claim 1.
\end{proof}

  Now to prove (4) given the above claim, let $D_0$ be a computable
  infinite set
  disjoint from $\ran(c_m)$ for every $m$. So $f(x) \restriction D_0 = g(x)
  \restriction D_0$. Then assuming
  that $\bigoplus_{i \in \omega} f(x_i) \leq_m f(x)$, we also have that
  $\bigoplus_{i \in \omega} (f(x_i) \inters D_0) \leq_m f(x)$, and hence
  $\bigoplus_{i \in \omega} (g(x_i) \inters D_0) \leq_m f(x)$. 
  But then by the Claim
  3, 
  there is a single $y \in [x]_E$ and a computable infinite set $B \subset
  \bigoplus_{i \in \omega} D_0$ so $B \notin
  I_2$ so that 
  $\bigoplus_{i \in \omega}(g(x_i) \inters D_0) \restriction B
  \leq_m g(y)$. Taking some $i$ so that $B^{[i]}$ is infinite and $x_i \neq
  y$ gives a contradiction since $g$ maps to a set of mutual
  $1$-generics.
\end{proof}

Hence, we have the following corollaries
\begin{cor}\label{m_cor}
  Suppose $\equiv_T$ is hyper-recursively-finite. Then
  \begin{enumerate}
    \item Conjecture~\ref{conj:strong_steel_for_m} is false. There is a Borel
    $(\equiv_T,\equiv_m)$-invariant function that is not uniformly invariant
    on any pointed perfect set.
    %Furthermore, for all $x$,
    %$f(x)$ is $\leq_m$-incomparable with $x$ and $\overline{x}$.
    \item \cite[Conjecture 1.1]{M} is false. That is, there is a universal countable Borel equivalence relation which
    is not uniformly universal.
    In particular, $\equiv_m$ and $\equiv_1$ on $\cantor$ are universal countable
    Borel equivalence relations.
    \item There are Borel $(\equiv_T,\equiv_m)$-invariant functions $(f_n)_{n \in \omega}$ on $\cantor$ so that for all $n,m$ and all $x \in \cantor$, $f_n(x)$ and $f_m(x)$ are $\leq_m$-incomparable, and for all $n$, $f_n(x)$ is $\leq_m$-incomparable with the Turing jump $j(x)$ and its complement $\overline{j}(x)$. 
  \end{enumerate}
\end{cor}
\begin{proof}
To prove (1), 
let $f$ be as in Theorem~\ref{m_construction} for the equivalence relation
$E = \equiv_T$.
Let $\Phi_e \from \cantor \to \cantor$ be a
total Turing functional with inverse $\Phi_d \from \cantor \to \cantor$
such that $x, \Phi_e(x), \Phi_e^2(x), \ldots$ are all distinct and have
the same Turing degree. Then if $f$ was uniformly
$(\equiv_T,\equiv_m)$-invariant it would contradict condition (4) of 
Theorem~\ref{m_construction}. 

To prove (2), note first that 
if every countable Borel equivalence relation $E$ is
hyper-Borel-finite, the function $f$ given in Theorem~\ref{m_construction} is a
Borel reduction from $E$ to many-one equivalence $\equiv_m$ and one-one
equivalence $\equiv_1$ on $2^\omega$. 
However, $\equiv_m$ is not uniformly universal by 
\cite[Theorem 1.5.(5)]{M}, so not every universal countable Borel
equivalence relation is uniformly universal. 

(3) follows by inductively making the functions $(f_n)_{n \in \omega}$ using
Theorem~\ref{m_construction}.(3) so $f_{n+1}(x)$ is $\leq_m$-incomparable with
$j(x)$, $\overline{j}(x)$ and $f_m(x)$ for $m \leq n$.
\end{proof}

We remark that it seems it should be straightforward to modify the above
construction in (3) to make a continuum size antichain of functions under
$\leq_m$ instead of a countable
antichain. However, making an infinite descending sequence of functions seems
more difficult.

It is open if $\equiv_T$ being hyper-recursively-finite implies that there is a
counterexample to Conjecture~\ref{conj:M_for_m}. That is, whether the
$\leq_m$-incomparability in Corollary~\ref{m_cor}.(3) can be improved to
incomparability under $\leq_m^\cone$.

It is open whether there is a counterexample to
Martin's conjecture or Steel's conjecture assuming $\equiv_T$ is hyper-recursively-finite:
\begin{question}
  Assume $\equiv_T$ is hyper-recursively-finite. Is Martin's conjecture
  false? Is Steel's conjecture false?  
\end{question}

\section{Open questions}
We pose a conjecture which would give a negative
answer to Question~\ref{hyp-rec-fin}. It states in a strong way that
Turing equivalence cannot be nontrivially written as an increasing union
of Borel equivalence relations. 

\begin{conj}\label{conj:non_approx}
  Suppose we write Turing equivalence as an increasing union
  \mbox{$(\equiv_T) =
  \bigunion_n E_n$} of Borel equivalence
  relations $E_n$ where $E_{n} \subset E_{n+1}$ for all $n$. Then there
  exists a pointed perfect set $P$ and some
  $i$ so that $E_i \restriction P = (\equiv_T \restriction P)$.
\end{conj}
 
In the context of probability measure preserving
equivalence relations, an analogous phenomenon of non-approximability has
been proved by Gaboriau and Tucker-Drob \cite{GTD}, e.g. for
pmp actions of property (T) groups. 

We know that Conjecture~\ref{conj:non_approx} implies some consequences of Martin's
conjecture. In particular, Conjecture~\ref{conj:non_approx} implies that Martin measure is $E_0$-ergodic in the
sense of \cite{T}.

\begin{prop}
  Suppose Conjecture~\ref{conj:non_approx} is true. Then if $f \from
  \cantor \to \cantor$ is a Borel
  homomorphism from Turing equivalence to $E_0$, i.e. $x \equiv_T y
  \implies f(x) E_0 f(y)$, then the $E_0$-class of $f(x)$ is 
  constant on a Turing cone.
\end{prop}
\begin{proof}
  Let $E_n$ be the subequivalence relation of $\equiv_T$ defined by $x
  \mathrel{E_n} y$ if $x \equiv_T y$ and $\forall k \geq n (f(x)(k) =
  f(y)(k))$. That is the $f(x)$ and $f(y)$ are equal past the first $n$
  bits. By Conjecture~\ref{conj:non_approx}, there is some $i$ and
  some pointed perfect set $P$ such that $E_i \restriction P = (\equiv_T
  \restriction P)$. Then by \cite{MSS}[Lemma 3.5] there is some pointed
  perfect set $P' \subset P$ such that for $x, y \in P'$, if $x \equiv_T y$,
  then $f(x) = f(y)$. Define $f'(x) = f(y)$ if there is $y \in P'$ such
  that $x \equiv_T y$, and $f'(x) = \emptyset$ otherwise. Thus, $f' \from
  \cantor \to \cantor$ is such that if $x \equiv_T y$, then
  $f'(x) = f'(y)$. Now any homomorphism from $\equiv_T$ to equality must be constant on a Turing cone, so 
  $f'$ is
  constant on a Turing cone. This implies the $E_0$-class of $f$ is constant on a
  cone.
\end{proof}

It is open if Conjecture~\ref{conj:non_approx} implies Martin's conjecture.

\begin{question}
  Assume Conjecture~\ref{conj:non_approx} is true. Does this imply Martin's
  conjecture for Borel functions?
\end{question}

The following is a diagram of some open questions surrounding
Martin's conjecture. All relationships between these open problems which
are not indicated by arrows are open. Note that Conjecture~\ref{conj:M_for_m} implies Martin measure is $E_0$-ergodic by following the proof of \cite{T}. Any homomorphism $f$ from $\equiv_T$ to $E_0$ is also a $(\equiv_T,\equiv_m)$-invariant function. So if it is not constant on a cone, then by Conjecture~\ref{conj:M_for_m} $f(x) \geq_m x$ on a cone. Then on a pointed perfect set $P$, there is a single many-one reduction $\rho$ so $f(x) \geq_m x$ via $\rho$ on $P$ so $f$ is injective on $P$. But if there is an injective homomorphism from a countable Borel equivalence relation $E$ to $E_0$, then $E$ is hyperfinite, and $\equiv_T$ is not hyperfinite on any pointed perfect set. 

\begin{center}
\begin{tikzpicture}[scale=.9]
\tikzstyle{every node}=[font=\footnotesize]
  \node[text width=2.5in,align=center] at (0,2.5) {Strong Steel's conjecture
  for $(\equiv_T,\equiv_m)$,
  Conjecture~\ref{conj:strong_steel_for_m}};

  \node[text width=2in,align=center] at (0,1) {Steel's conjecture
  \cite[Conjecture III]{SS}};

  \node[text width=2in,align=center] at (0,-.5) {Martin's conjecture \cite[Conjecture I, II]{SS}};

  \node[text width=2in,align=center] at (0,-2) {Martin measure is
  $E_0$-ergodic \cite{T}};

  \node[text width=2in,align=center] at (0,-3.5) {$\equiv_T$ is not Borel
  bounded \cite{BJ}};

  \node[text width=2.4in,align=center] at (5,1) {Steel's conjecture
  for \mbox{$(\equiv_T,\equiv_m)$},
  Conjecture~\ref{conj:steel_for_m}};

  \node[text width=2.1in,align=center] at (5,-.5) {Martin's conjecture
  for $(\equiv_T,\equiv_m)$,
  Conjecture~\ref{conj:M_for_m}};

  \node[text width=1.8in,align=center] at (-5,-.5) {$\equiv_T$ is
  non-approximable, Conjecture~\ref{conj:non_approx}};

  \node[text width=1.5in,align=center] at (-5,1) {$\equiv_T$ is not
  hyper-recursively-finite};

  \node[text width=1.5in,align=center] at (-5,2.5) {Every universal countable
  Borel equivalence relation is uniformly universal \cite{M}};

  \draw [->] (0,2) -- (0,1.5);
  \draw [->] (0,.5) -- (0,0);
  \draw [->] (-5,2) -- (-5,1.5);
  \draw [->] (0,-1) -- (0,-1.5);
  \draw [->] (0,-2.5) -- (0,-3);
  \draw [->] (5,.5) -- (5,0);
  \draw [->] (1.5,2) -- (3.5,1.5);
  \draw [->] (-1.5,2) -- (-3.5,1.5);
  \draw [->] (-5,0) -- (-5,.5);
  \draw [->] (-3.5,-1) -- (-1.5,-1.5);
  \draw [->] (3.5,-1) -- (1.5,-1.5);

\end{tikzpicture}
\end{center}

\end{document}